\documentclass[11pt, reqno]{amsart}
\usepackage{geometry}                 
\usepackage{graphicx}
\usepackage{amssymb,latexsym,amsmath}
\usepackage{amsfonts,amsthm}
\usepackage{amscd}
\usepackage{mathrsfs}
\usepackage{cite,upref}
\usepackage{eucal}
\usepackage{epstopdf}
\usepackage{amsgen}
\usepackage{xspace}
\usepackage{verbatim}

\newcommand{\dist}{\operatorname{dist}}
\newcommand{\gd}{\Delta}

\newcommand{\inpt}[1]{\langle #1 \rangle}

\newcommand{\mX}{\mathcal{X}}
\newcommand{\mY}{\mathcal{Y}}

\newcommand{\mU}{\mathcal{U}}
\newcommand{\mV}{\mathcal{V}}
\newcommand{\mW}{\mathcal{W}}

\newcommand{\ms}{\mathscr}

\newcommand{\gw}{\Omega}
\newcommand{\ga}{\gamma}

\newcommand{\gl}{\lambda}
\newcommand{\gL}{\Lambda}

\newcommand{\gk}{\kappa}
\newcommand{\om}{\omega}
\newcommand{\gz}{\theta}
\newcommand{\nb}{\nabla}
\newcommand{\vp}{\varphi}

\newcommand{\tup}{\textup}
\newcommand{\csg}{\{ S(t)\}_{t\geq0}}

\newcommand{\beq}{\begin{equation}}
\newcommand{\eeq}{\end{equation}}
\newcommand{\bl}{\label}
\numberwithin{equation}{section}

\theoremstyle{plain}
\newtheorem{theorem}{Theorem}
\newtheorem{lemma}{Lemma}

\newtheorem{proposition}{Proposition}
\theoremstyle{definition}
\newtheorem{definition}{Definition}
\theoremstyle{remark}

\begin{document}

\title[OREGONATOR SYSTEM]{\textbf{Longtime Dynamics of the Oregonator System}}
\author{Yuncheng You} 
\address{Department of Mathematics and Statistics \\
University of South Florida \\
Tampa, FL 33620} 
\email{you@mail.usf.edu} 
\keywords{Reaction-diffusion system, Oregonator, global attractor, absorbing set, asymptotic compactness, exponential attractor}
\subjclass[2000]{37L30, 35B40, 35B41, 35K55, 35K57, 80A32, 92B05.}
\date{}

\begin{abstract}
In this work the existence and properties of a global attractor for the solution semiflow of the Oregonator system are proved. The Oregonator system is the mathematical model of the famous Belousov-Zhabotinskii reaction. A rescaling and grouping estimation method is developed to show the absorbing property and the asymptotic compactness of the solution trajectories of this three-variable reaction-diffusion system with quadratic nonlinearity from the autocatalytic kinetics. It is proved that the fractal dimension of the global attractor is finite. The existence of an exponential attractor for this Oregonator semiflow is also shown.
\end{abstract}
\maketitle

\section{\textbf{Introduction}}

The Belousov-Zhabotinskii (BZ) reaction is a class of oxidation reactions of organic components catalyzed by bromate ions, which exhibits oscillatory phenomena. The temporal oscillations of the reaction was first reported by B.P. Belousov in 1958 and the development of oscillatory spatial structures was reported later by Zhabotinskii in 1967, cf. \cite{aZ91}. Since then the BZ reaction has been extensively studied by physical chemists on its kinetic behavior \cite{EP98, PL68, jT82} and by mathematicians on the dynamics and patterns of the solutions of the associated mathematical model \cite{FB85, FN74, cvP88, RP92, TF80}.

Field, K\"{o}r\"{o}s, and Noyes \cite{FKN72} developed a detailed chemical mechanism for the BZ reaction and later Field and Noyes \cite{FN74} derived a simplified model  as a reaction-diffusion system (originally an ODE system) consisting of three unknowns, which retains most of the important features of the FKN mechanism. This Field-Noyes model is called Oregonator, the name coined by J.J. Tyson \cite{jT76}, which refers to the working place of the two scholars. 

The chemical reaction scheme of this Field-Noyes model is given by
$$
	\tup{A + Y} \longrightarrow \tup{X}, \quad  \tup{X + Y} \longrightarrow \tup{P}, \quad \tup{B + X} \longrightarrow \tup{2X + Z}, \quad \tup{2X} \longrightarrow \tup{Q}, \quad \tup{Z} \longrightarrow \tup{$\gk$Y},
$$
where \tup{A} and \tup{B} are reactants, \tup{P} and \tup{Q} are products, $\gk$ is a stoichiometric constant, and \tup{X, Y}, and \tup{Z} are the three key intermediate substances representing \tup{HBrO$_2$} (bromous acid), \tup{Br$^{-}$} (bromide ion), and \tup{Ce(IV)} (Cerium), respectively. Under the assumption that the concentrations of the reactants \tup{A} and \tup{B} as well as the catalytic H$^+$ ion are held constant, the dimensionless form of the diffusive Oregonator system is given by

\begin{align} 
	\frac{\partial u}{\partial t} &= d_1 \gd u + a_1 u + b_1 v - Fu^2 - G_1 uv,  \bl{eu} \\
	\frac{\partial v}{\partial t} &= d_2 \gd v - b_2 v + c_2 w - G_2 uv,  \bl{ev} \\
	\frac{\partial w}{\partial t} &= d_3 \gd w + a_3 u - c_3 w, \bl{ew}
\end{align}
where $u (t, x), v(t, x)$ and $w(t, x)$ represent the concentrations of \tup{X}, \tup{Y}, and \tup{Z}, respectively, for $t > 0, x \in \gw$, and $\gw$ is a bounded Lipschitz domain in $\mathbb{R}^n\, (n \leq 3)$, with the homogeneous Dirichlet boundary condition 
\begin{equation} \label{dbc}
	u(t, x) = v(t, x) = w (t, x) = 0, \quad t > 0, \; \, x \in \partial \gw,
\end{equation}
and an initial condition
\begin{equation} \label{ic}
        u(0,x) = u_0 (x), \; v(0, x) = v_0 (x), \; w(0,x) = w_0 (x),\quad x \in \gw.
\end{equation}
The diffusive coefficients $d_1, d_2, d_3$ and the reaction rate constants $a_i^\prime$s, $b_i^\prime$s, $c_i^\prime$s, $F$ and $G_i^\prime$s are all positive constants. In this work, we shall study the asymptotic dynamics of the solution semiflow generated by this problem. 

The diffusive Oregonator system is a prototype of many \emph{quadratic autocatalytic} reaction-diffusion systems served as mathematical models in physical chemistry and in mathematical biology, especially the kinetic biochemical reactions in cell and molecular biology. For the \emph{cubic autocatalytic} reaction-diffusion systems such as the Brusselator system \cite{PL68}, Gray-Scott equations \cite{GS83, GS84}, Schnackenberg equations \cite{jS79}, and Selkov equations \cite{eS68}, after the seminal publications \cite{LMOS93, jP93} there have been extensive studies by numerical simulations and by mathematical analysis on spatial patterns (including but not restricted to Turing patterns) and complex bifurcations as well as asymptotic dynamics, see the references in \cite{yY07, yY08, yY09a, yY09b, yY10}.

For the Oregonator system, the global existence of classical solutions in the continuously differentiable function spaces and the stability of steady-state positive steady-state solutions are studied in \cite{TF80, cvP88, RP92}. In \cite[Section II.4.4]{CV02} it is shown that under the condition 
\beq \bl{rst}
	c_2^2 < 4 b_2 c_3,
\eeq 
there exists a solution semiflow for the evolutionary equations formulated from \eqref{eu}--\eqref{ew} with the homogeneous Dirichlet or Neumann boundary conditions in the positive invariant region of the product $L^2$ space and that solutiion semiflow has a global attractor.

In this paper, we shall remove the rerstriction \eqref{rst} and prove the existence of a global attractor in the product $L^2$ phase space for the semiflow of the weak solutions of the Oregonator system \eqref{eu}--\eqref{ew} with the homogeneous Dirichlet boundary conditions \eqref{dbc}. The results are also valid for the corresponding Neumann boundary conditions. We shall also study the properties and the fractal dimension of the global attractor. Furthermore, we shall prove the existence of an exponential attractor for this solution semiflow. %especially the entropy as a characteristic index of the chaotic behavior of the dynamics associated with the global attractor.

For most reaction-diffusion systems consisting of two or more equations arising from the scenarios of autocatalytic chemical reactions or biochemical activator-inhibitor reactions, such as the Brusselator system and the Oregonator systems here, the \emph{asymptotically dissipative sign condition} in vector version,
$$
	\lim_{|s| \to \infty} f(s) \cdot s \leq C,
$$
where $C \geq 0$ is a constant, is inherently not satisfied by the nonlinear part of the equations, see \eqref{opF} later. Besides there is a \emph{coefficient barrier} caused by the arbitrary and different coefficients in the linear terms of the three equations. These are the obstacles for showing the absorbing property and the asymptotically compact property of the semiflow of the weak solutions of the Oregonator system.

The novel feature in this paper is to overcome these obstacles and to make the \emph{a priori} estimates by a method of \emph{rescaling and grouping estimation} that proves the globally dissipative and attractive longtime dynamics in terms of the existence of a global attractor.

We start with the formulation of an evolutionary equation associated with the initial-boundary value problem \eqref{eu}--\eqref{ic}. Define the product Hilbert spaces as follows,
\begin{equation*}
    	H = [L^2 (\gw)]^3 , \quad E =  [H_{0}^{1}(\gw)]^3, \quad \textup{and} \quad \Pi =  [(H_{0}^{1}(\gw) \cap H^{2}(\gw))]^3.
\end{equation*}
The norm and inner-product of $H$ or the component space $L^2 (\gw)$ will be denoted by
$\| \, \cdot \, \|$ and $\inpt{\,\cdot , \cdot\,}$, respectively. The norm of $L^p (\gw)$ will be denoted by $\| \, \cdot \, \|_{L^{p}}$ if $p \ne 2$. By the Poincar\'{e} inequality and the homogeneous Dirichlet boundary condition \eqref{dbc}, there is a constant $\ga > 0$ such that
\begin{equation} \label{pcr}
    \| \nabla \varphi \|^2 \geq \ga \| \varphi \|^2, \quad \textup{for} \;  \varphi \in H_{0}^{1}(\gw) \; \textup{or} \;  E,
\end{equation}
and we shall take $\| \nabla \varphi \|$ to be the equivalent norm $\| \varphi \|_E$ of the space $E$ or the component space $H_{0}^{1}(\gw)$. We use $| \, \cdot \, |$ to denote an absolute value or a vector norm in a Euclidean space.

It can be checked easily that, by the Lumer-Phillips theorem and the analytic semigroup generation theorem \cite{SY02}, the linear differential operator
\begin{equation} \label{opA}
        A =
        \begin{pmatrix}
            d_1 \gd     & 0    &0  \\[3pt]
            0 & d_2 \gd    &0  \\[3pt]
            0 &0  &d_3 \gd
        \end{pmatrix}
        : D(A) (= \Pi) \longrightarrow H
\end{equation}
is the generator of an analytic $C_0$-semigroup on the Hilbert space $H$, which will be denoted by $\{e^{At}, t \geq 0\}$. By the fact that $H_{0}^{1}(\gw) \hookrightarrow L^6(\gw) \hookrightarrow L^4(\gw)$ is a chain of continuous embeddings for $n \leq 3$ and using the H\"{o}lder inequality, 
$$
	\| u v \| \leq \| u \|_{L^4} \| v \|_{L^4}, \quad \| u^{2}\| = \| u \|_{L^4}^2, \quad  \textup{for} \; u, v \in L^4 (\gw),
$$
one can verify that the nonlinear mapping
\begin{equation} \label{opF}
    f(g) =
        \begin{pmatrix}
            a_1 u + b_1 v - Fu^2 - G_1 u v  \\[3pt]
            - b_2 v  + c_2 w - G_2 u v \\[3pt]
            a_3 u - c_3 w  
        \end{pmatrix}
        : E \longrightarrow H,
\end{equation}
where $g = (u, v, w)$, is well defined on $E$ and the mapping $f$ is locally Lipschitz continuous. Thus the initial-boundary value problem \eqref{eu}--\eqref{ic} is formulated into an initial value problem of the Oregonator evolutionary equation,
\begin{equation} \label{eveq} 
    \begin{split}
    \frac{dg}{dt} &= A g + f(g), \quad t > 0, \\[2pt]
     g(0) &= g_0 = \textup{col} \, (u_0, v_0, w_0).
     \end{split}
\end{equation}
where $g (t) = \textup{col} \, (u(t, \cdot), v(t, \cdot), w(t, \cdot))$, simply written as $(u(t, \cdot), v(t, \cdot), w(t, \cdot))$. We shall accordingly write $g_0 = (u_0, v_0, w_0)$. 

The following proposition will be used in proving the existence of a weak solution to this initial value problem. Its proof is seen in \cite[Theorem II.1.4]{CV02} and in \cite[Proposition I.3.3]{BP86}.
\begin{proposition} \label{P1} 
Consider the Banach space
\begin{equation} \label{wsp}
	W(0,\tau) = \left\{\zeta (\cdot): \zeta \in L^2 (0, \tau; E) \; \textup{and}\; \partial_t \zeta \in L^2 (0, \tau; E^*)\right\}
\end{equation}
with the norm
$$
	\|\zeta \|_W = \|\zeta \|_{L^2 (0, \tau; E)} + \|\partial_t \zeta \|_{L^2 (0, \tau; E^*)}.
$$
Then the following statements hold:

\textup{(a)} The embedding $W(0, \tau) \hookrightarrow L^2 (0, \tau; H)$ is compact.

\textup{(b)} If $\zeta \in W(0, \tau)$, then it coincides with a function in $C([0, \tau]; H)$ for a.e. $t \in [0, \tau]$.

\textup{(c)} If $\zeta, \xi \in W(0, \tau)$, then the function $t \to \inpt{\zeta (t), \xi (t)}_H$ is absolutely continuous on $[0, \tau]$ and
$$
	\frac{d}{dt} \inpt{\zeta (t), \xi (t)} = \left(\frac{d\zeta}{dt}, \xi (t)\right) + \left(\zeta(t), \frac{d\xi}{dt}\right), \; a.e. \, t \in [0, \tau],
$$
where $(\cdot , \cdot)$ is the $(E^*, E)$ dual product.
\end{proposition}
By conducting \emph{a priori} estimates on the Galerkin approximate solutions of the initial value problem \eqref{eveq} and through extracting the weak and weak* convergent subsequences in the appropriate spaces, we can prove the local and then global existence and uniqueness of the weak solution $g(t)$ of \eqref{eveq} in the next section, also the continuous dependence of the solutions on the initial data and the regularity properties satisfied by the weak solution. Therefore, the weak solutions for all initial data in $H$ form a semiflow in the space $H$.

We refer to \cite{jH88, SY02, rT88} and many references therein for the concepts and basic facts in the theory of infinite dimensional dynamical systems.
\begin{definition} \label{D:abs}
Let $\{S(t)\}_{t \geq 0}$ be a semiflow on a Banach space $\mX$. A bounded subset $B_0$ of $\mX$ is called an \emph{absorbing set} in $\mX$ if, for any bounded subset $B \subset \mX$, there is a finite time $t_0 \geq 0$ depending on $B$ such that $S(t)B \subset B_0$ for all $t \geq t_0$.
\end{definition}

\begin{definition} \label{D:asp}
A semiflow $\{S(t)\}_{t \geq 0}$ on a Banach space $\mX$ is called \emph{asymptotically compact} if for any bounded sequences $\{x_n \}$ in $\mX$ and $\{t_n \} \subset (0, \infty)$ with $t_n \to \infty$, there exist subsequences $\{x_{n_k}\}$ of $\{u_n \}$ and $\{t_{n_k}\}$ of $\{t_n\}$, such that $\lim_{k \to \infty} S(t_{n_k})x_{n_k}$ exists in $\mX$.
\end{definition}

\begin{definition} \label{D:atr}
Let $\{S(t)\}_{t \geq 0}$ be a semiflow on a Banach space $\mX$. A subset $\ms{A}$ of $\mX$ is called a \emph{global attractor} for this semiflow, if the following conditions are satisfied: 

(i) $\ms{A}$ is a nonempty, compact, and invariant subset of $\mX$ in the sense that 
$$
	S(t)\ms{A} = \ms{A} \quad \textup{for any} \; \;  t \geq 0. 
$$

(ii) $\ms{A}$ attracts any bounded set $B$ of $\mX$ in terms of the Hausdorff distance, i.e.
$$
	\text{dist} (S(t)B, \ms{A}) = \sup_{x \in B} \inf_{y \in \ms{A}} \| S(t)x - y\|_{\mX} \to 0, \;\; \text{as} \; \, t \to \infty.
$$
\end{definition}

The following proposition states concisely the basic result on the existence of a global attractor for a semiflow, cf. \cite{jH88, SY02, rT88}.

\begin{proposition} \label{P:exga}
Let $\csg$ be a semiflow on a Banach space or an invariant region $\mX$ in it. If the following conditions are satisfied\textup{:} 

\textup{(i)} $\csg$ has a bounded absorbing set $B_0$ in $\mX$, and 

\textup{(ii)} $\csg$ is asymptotically compact in $\mX$, \\
then there exists a global attractor $\ms{A}$ in $\mX$ for this semiflow, which is given by
$$
    \ms{A} = \om (B_0) \overset{\textup{def}}{=} \bigcap_{\tau \geq 0} \text{Cl}_{\mX}  \bigcup_{t \geq \tau} (S(t)B_0).
$$
\end{proposition}
In Section 2 we prove the local existence and uniqueness of the weak solutions of the Oregonator evolutionary equation \eqref{eveq} and in Section 3 we shall prove the global existence of the weak solutions and the absorbing property of this solution semiflow. In Section 4 we shall prove the asymptotic compactness of this solutions semiflow and show the existence of a global attractor in the space $H$ for this Oregonator semiflow. In Section 5 we show that the global attractor has a finite Hausdorff dimension and a finite fractal dimension. In Section 6 we prove some regularity properties of this global attractor. Finally, in Section 7, the existence of an exponential attractor for this solution semiflow is shown.

\section{\textbf{The Local Existence of Weak Solutions}}

In this paper, we shall write $u(t, x), v(t, x), w(t, x)$ simply as $u(t), v(t), w(t)$ or even as $u, v, w$. Similarly for other functions of $(t,x)$. 

The local existence and uniqueness of the solution to a system of multi-component reaction-diffusion equations such as the IVP \eqref{eveq} with certain regularity requirement is not a trivial issue. There are two different approaches to get a solution in a Sobolev space. One approach is the mild solutions provided by the "variation-of-constant formula" in terms of the associated linear semigroup $\{e^{At}\}_{t\geq 0}$, but the parabolic theory of mild solutions requires that $g_0 \in E$ instead of $g_0 \in H$ assumed here. The other approach is the weak solutions obtained through the Galerkin approximations and the Lions-Magenes type of compactness treatment, cf. \cite{CV02, jL69}. 
\begin{definition}
A function $g(t,x), (t, x) \in [0, \tau] \times \gw$, is called a weak solution to the IVP of the parabolic evolutionary equation \eqref{eveq}, if the following two conditions are satisfied:

(i) $\frac{d}{dt} (g, \phi) = (Ag, \phi) + (f(g), \phi)$ is satisfied for a.e. $t \in [0, \tau]$ and for any $\phi \in E$; 

(ii) $g (t, \cdot) \in L^2 (0, \tau; E) \cap C_w ([0, \tau]; H)$ such that $g (0) = g_0$.

\noindent
Here $(\cdot, \cdot)$ stands for the $(E^*, E)$ dual product.
\end{definition}
For reaction-diffusion systems with more general nonlinearity that may be of higher degrees and may involve transport terms, the definition of corresponding weak solutions is made in \cite[Definition XV.3.1]{CV02}.

\begin{lemma} \label{L:locs}
For any given initial datum $g_0 \in H$, there exists a unique, local, weak solution $g(t) = (u(t), v(t), w(t)), \, t \in [0, \tau]$ for some $\tau > 0$, of the Oregonator evolutionary equation \eqref{eveq} such that $g(0) = g_0$, which satisfies 
\begin{equation} \label{soln}
    g \in C([0, \tau]; H) \cap C^1 ((0, \tau); H) \cap L^2 (0, \tau; E).
\end{equation}
\end{lemma}

\begin{proof}
Using the orthonormal basis of eigenfunctions $\{e_j (x)\}_{j=1}^\infty$ of the Laplace operator with the homogeneous Dirichlet boundary condition:
$$
	\gd e_j + \gl_j e_j = 0 \;\; \textup{in} \;\, \gw, \quad e_j |_{\partial \gw} = 0,  \quad j = 1, 2, \cdots, n, \cdots,
$$
we consider the solution 
\begin{equation} \label{gq}
	g_m (t, x) = \sum_{j=1}^m q_j^m (t) e_j (x), \quad t \in [0, \tau], \; x \in \gw,
\end{equation}
of the approximate system
\begin{equation} \label{eveqm} 
	\begin{split}
    \frac{\partial g_m}{\partial t} = A &g_m + P_m f(g_m), \quad t > 0, \\[3pt]
      g_m(0) &= P_m \, g_0 \in H_m,
      \end{split}
\end{equation}
where each $q_j^m (t)$ for $j = 1, \cdots, m$ is a three-dimensional vector function of $t$ only, corresponding to the three unknowns $u, v$, and $w$, and $P_m: H \to H_m = \textup{Span} \{e_1, \cdots, e_m\}$ is the orthogonal projection. Note that for each given integer $m \geq 1$, \eqref{eveqm} can be written as an IVP of a system of ODEs, whose unknown is a $3m$-dimensional vector function of all the coefficient functions of time $t$ in the expansion of $g_m (t,x)$, namely,
$$
	q^m (t) = \textup{col} \, (q_{ju}^m (t), q_{jv}^m (t), q_{jw}^m (t); \, j = 1, \cdots, m). 
$$
The IVP of this ODE system \eqref{eveqm} can be written as 
\begin{equation} \label{odeg}
	\begin{split}
	\frac{dq^m}{dt} &= \gL_m q^m (t) + f_m (q^m (t)), \; t > 0, \\[3pt]
	q^m (0) &= \textup{col} (P_m u_{0j}, \, P_m v_{0j}, \, P_m w_{0j}; \, j = 1, \cdots, m).
	\end{split}
\end{equation}
Note that $\gL_m$ is a matrix and $f_m$ is a $3m$-dimensional vector of quadratic polynomials of $3m$-variables, which is certainly a locally Lipschitz continuous vector function in $\mathbb{R}^{3m}$. Thus the solution of the initial value problem \eqref{odeg} exists uniquely on a time interval $[0, \tau]$, for some $\tau > 0$. Substituting all the components of this solution $q^m (t)$ into \eqref{gq}, we obtain a unique local solution $g_m (t, x)$ of the initial value problem \eqref{eveqm}, for any $m \geq 1$.

By the multiplier method we can conduct \emph{a priori} estimates based on 
\begin{equation*} 
	\frac{1}{2} \| g_m (t) \|_{H_m}^2 + \inpt{\boldsymbol{d} \nabla g_m (t), \nabla g_m (t)}_{H_m} = \inpt{P_m f(g_m (t)), g_m (t)}_{H_m}, \quad t \in [0, \tau],
\end{equation*}
where $\boldsymbol{d} = \textup{diag} (d_1, d_2, d_3)$ is a diagonal matrix. These estimates are similar to what we shall present in Lemma \ref{L:glsn} in the next section. Note that $\|g_m (0)\| = \|P_m g_0 \| \leq \|g_0 \|$ for all $m \geq 1$.  It follows that (see the proof of Lemma \ref{L:glsn})
$$
	\{g_m\}_{m=1}^\infty \, \textup{is a bounded sequence in} \, L^2 (0, \tau; E) \cap L^\infty (0, \tau; H),
$$
and, since $A: E \to E^*$ is a bounded linear operator, 
$$
	\{Ag_m\}_{m=1}^\infty \, \textup{is a bounded sequence in} \, L^2 (0, \tau; E^*), \; \; \textup{where} \, E^* = [H^{-1} (\gw)]^3. 
$$
Since $f: E \to H$ is continuous, 
\begin{equation} \label{bdf}
	\{P_m f(g_m)\}_{m=1}^\infty \, \textup{is a bounded sequence in} \, L^2 (0, \tau; H) \subset L^2 (0, \tau; E^*).
\end{equation}
Therefore, by taking subsequences (which we will always relabel as the same as the original one), there exist limit functions 
\begin{equation} \label{limgf}
	g(t, \cdot) \in L^2(0, \tau; E) \cap L^\infty (0, \tau; H) \quad \textup{and} \quad \Phi (t, \cdot) \in L^2 (0, \tau; H)
\end{equation}
such that 
\begin{equation} \label{limg}
	\begin{split}
	g_m &\longrightarrow g \;\, \textup{weakly in} \; L^2 (0, \tau; E),   \\
	g_m &\longrightarrow g \;\, \textup{weak* in} \; L^\infty (0, \tau; H),   \\
	Ag_m &\longrightarrow Ag \;\, \textup{weakly in} \; L^2 (0, \tau; E^*),  
	\end{split}
\end{equation}
and
\begin{equation} \label{limf}
		P_m f(g_m) \longrightarrow \Phi \;\, \textup{weakly in} \; L^2 (0, \tau; H), 
\end{equation}
as $m \to \infty$. To estimate the (distributional) time derivative sequence $\{\partial_t g_m\}_{m=1}^\infty$, we take the supremum of the $(E^*, E)$ dual product of the equation \eqref{eveqm} with any $\eta \in E$ and use the fact that $(h, \eta) = \inpt{h, \eta}$ for any $h \in H$ to obtain
$$
	\|\partial_t g_m (t) \|_{E^*} \leq C \left(\|A g_m (t)\|_{E^*} + \|P_m f(g_m (t))\|_H \right), \quad t \in [0, \tau], \, m \geq 1,
$$
where $C$ is a uniform constant for all $m \geq 1$. Thus by the boundedness in \eqref{limg} and \eqref{limf} it holds that
$$
	\{\partial_t g_m\}_{m=1}^\infty \, \textup{is a bounded sequence in} \, L^2 (0, \tau; E^*),
$$
and, by further extracting of a subsequence if necessary, it follows from the uniqueness of distributional time derivative that 
\begin{equation} \label{limt}
	\partial_t g_m \longrightarrow \partial_t g \;\, \textup{weakly in} \; L^2 (0, \tau; E^*),  \, \textup{as} \; m \to \infty.
\end{equation}

In order to show that the limit function $g$ is a weak solution to the IVP \eqref{eveq}, we need to show $\Phi = f(g)$. By Proposition \ref{P1}, item (a), the boundedness of $\{g_m\}$ in $L^2 (0, \tau; E)$ and $\{\partial_t g_m\}$ in $L^2 (0, \tau; E^*)$ implies that (in the sense of subsequence extraction)
\begin{equation} \label{strh}
	g_m  \longrightarrow g \;\, \textup{strongly in} \; L^2 (0, \tau; H),  \, \textup{as} \; m \to \infty.
\end{equation}
Consequently, there exists a subsequence such that
\begin{equation} \label{aeg}
	g_m (t, x) \longrightarrow g(t, x) \;\, \textup{for a.e.} \; (t, x) \in [0, \tau] \times \gw,  \, \textup{as} \; m \to \infty.
\end{equation}
Due to the continuity of the mapping $f$, we have
\begin{equation} \label{aef}
	f(g_m (t, x))  \longrightarrow f(g(t, x)) \;\, \textup{for a.e.} \; (t, x) \in [0, \tau] \times \gw,  \, \textup{as} \; m \to \infty.
\end{equation}
According to \cite[Lemma I.1.3]{jL69} or \cite[Lemma II.1.2]{CV02} and by the triangle inequality in terms of the $H$-norm, the two facts \eqref{bdf} and \eqref{aef} guarantee that 
\begin{equation} \label{limfg}
	P_m f(g_m) \longrightarrow f(g) \;\, \textup{weakly in} \; L^2 (0, \tau; H), \; \textup{as} \; m \to \infty.
\end{equation}
By the uniqueness, \eqref{limf} and \eqref{limfg} imply that $\Phi = f(g)$ in $L^2 (0, \tau; H)$.

With \eqref{limg}, \eqref{limt} and \eqref{limfg}, by taking limit of the integral of the weak version of \eqref{eveqm} with any given $\phi \in L^2 (0, \tau; E)$, 
$$
	\int_0^\tau \left(\frac{\partial g_m}{\partial t}, \phi \right) dt = \int_0^\tau \left[(Ag_m, \phi) + (P_m f(g_m), \phi)\right] dt, \;\; \textup{as} \; m \to \infty,
$$
we obtain
\begin{equation} \label{iw}
	\int_0^\tau \left(\frac{\partial g}{\partial t}, \phi\right) dt = \int_0^\tau \left[(Ag, \phi) + (f(g), \phi)\right] dt, \;\; \textup{for any} \; \phi \in L^2 (0, \tau; E).
\end{equation}
Let $\phi \in E$ be any constant function. Then we can use the property of Lebesgue points for each integral in \eqref{iw} to obtain
\begin{equation} \label{weq} 
	\frac{d}{dt} (g, \phi) = (Ag, \phi) + (f(g), \phi), \;  \textup{for a.e}. \; t \in [0, \tau] \; \textup{and  for any} \;  \phi \in E.
\end{equation}

Next, $\partial_t g \in L^2 (0, \tau; E^*)$ shown in \eqref{limt} implies that $g \in C_w ([0, \tau]; E^*)$. Since the embedding $H \hookrightarrow E^*$ is continuous, this $g \in C_w ([0, \tau]; E^*)$ and $g \in L^\infty (0, \tau; H)$ shown in \eqref{limg} imply that 
$$
	g \in C_w ([0, \tau]; H). 
$$
due to \cite[Theorem II.1.7 and Remark II.1.2]{CV02}. Now we show that $g(0) = g_0$. By Proposition \ref{P1}, item (c), for any $\phi \in C^1([0, \tau]; E)$ with $\phi (\tau) = 0$, we have
$$
	\int_0^\tau \left(- g, \partial_t \phi\right) \, dt = \int_0^\tau \left[(Ag, \phi) + (f(g), \phi)\right] dt + \inpt{g(0), \phi(0)},
$$
and
$$
	\int_0^\tau \left(- g_m, \partial_t \phi\right) \, dt = \int_0^\tau \left[(Ag_m, \phi) + (P_m f(g_m), \phi)\right] dt + \inpt{P_m g_0, \phi(0)}.
$$
Take the limit of the last equality as $m \to \infty$. Since $P_m g_0 \to g_0$ in $H$, we obtain $\inpt{g(0), \phi (0)} = \inpt{g_0, \phi (0)}$ for any $\phi (0) \in E$. Finally the denseness of $E$ in $H$ implies that $g(0) = g_0$ in $H$. Then by checking against Definition 4, we conclude that the limit function $g$ is a weak solution to the initial value problem \eqref{eveq}.

The uniqueness of weak solution can be shown by estimating the difference of any two possible weak solutions with the same initial value $g_0$ through the weak version of the evolutionary equation (or the variation-of-constant formula) and the Gronwall inequality. 

By Proposition \ref{P1}, item (b), and the fact that the weak solution $g \in W(0, \tau)$, the space defined in \eqref{wsp}, we see that $g \in C ([0, \tau]; H)$, which also infers the continuous dependence of the weak solution $g(t) = g(t; g_0)$ on $g_0$ for any $t \in [0, \tau]$. 

Moreover, since $g \in L^2 (0, \tau; E)$, for any $t \in (0, \tau)$ there exists an earlier time $t_0 \in (0, t)$ such that $g (t_0) \in E$. Then the weak solution coincides with the strong solution expressed by the mild solution on $[t_0, \tau]$, cf. \cite{CV02, SY02}, which turns out to be continuously differentiable in time at $t$ strongly in $H$, cf. \cite[Theorem 48.5]{SY02}. Thus we have shown $g \in C^1 ((0, \tau); H)$ and the weak solution $g$ satisfies the properties specified in \eqref{soln}. 
\end{proof}

\section{\textbf{Absorbing Properties}}

In this section, we shall prove the global existence of the weak solutions and investigate the absorbing properties of the solution semiflow.

The following proposition \cite[Theorem II.4.2]{CV02} provides the necessary and sufficient conditions for a semilinear parabolic system 
\beq \bl{sps}
	\frac{\partial \xi}{\partial t} = A_0 \, \gd \xi + \psi (\xi) + \gz (x), \quad t > 0, \; \; x \in \gw,
\eeq
on a bounded Lipschitz domain $\gw \subset \mathbb{R}^n$ with the homogeneous Dirichlet (or Neumann) boundary condition to have the positive cone $\mathbb{R}_{+}^N$ as an invariant region \cite{Sm83}. Here $A_0$ is an $N \times N$ symmetric and positive definite matrix, $\xi (t, x) = \tup{col}\, (\xi_1, \xi_2, \cdots \xi_N), \gz (x) = \tup{col}\, (\gz_1, \gz_2, \cdots, \gz_N) \in [L^2 (\gw)]^N$ is a given vector function, and $\psi = \tup{col}\, (\psi_1, \psi_2, \cdots, \psi_N): [H_0^1 (\gw)]^N (\tup{or} \, [H^1 (\gw)]^N) \to [L^2 (\gw)]^N$ is locally Lipschitz continuous.

\begin{proposition} \bl{P3}
The positive cone $\mathbb{R}_{+}^N = \{\xi \in \mathbb{R}^N: \xi_i \geq 0,\, i = 1, \cdots, N\}$ is an invariant region for \eqref{sps} if and only if the following two conditions are satisfied:

\textup{(i)}  $A_0$ is a diagonal matrix; and

\textup{(ii)} for every $i = 1, 2, \cdots, N$, it holds that
\beq  \bl{pcc}
	\psi_i (\xi_1, \, \cdots, \xi_{i-1}, \, 0, \, \xi_{i+1}, \, \cdots, \xi_N) + \gz_i (x) \geq 0, 
\eeq
for any $x \in \gw$ and $\xi_j \geq 0, j \ne i$.
\end{proposition}

Define $H_+$ and $E_+$ to be the cones in the space $H$ and $E$, respectively, as follows,
\begin{align*}
	H_+ &= \{\varphi (\cdot) = (\vp_1, \vp_2, \vp_3) \in H: \, \varphi_i (x) \geq 0, \, x \in \gw, \, i = 1, 2, 3\}, \\
	E_+ &= \{\varphi (\cdot) = (\vp_1, \vp_2, \vp_3) \in E: \, \varphi_i (x) \geq 0, \, x \in \gw, \, i = 1, 2, 3\}.
\end{align*}
\begin{lemma} \label{L:glsn}
For any initial datum $g_0 =(u_0, v_0, w_0) \in H_+$, there exists a unique, global, weak solution $g(t) = (u(t), v(t), w(t)), \, t \in [0, \infty)$, of the initial value problem of the Oregonator evolutionary equation \eqref{eveq} and it becomes a strong solution on the time interval $(0, \infty)$. Moreover, there exists an absorbing set $B_0$ in $H_+$,
\begin{equation} \label{absb}
    B_0=\left\{g\in H_+: \|g\|^2\leq K_1\right\},
\end{equation}
where $K_1$ is a positive constant, for the solution semiflow $\csg$ of this evolutionary equation \eqref{eveq} on the cone $H_+$.
\end{lemma}

\begin{proof}
It is easy to see that by Proposition \ref{P3} the cone $\mathbb{R}_+^3$ is an invariant region for the Oregonator system \eqref{eu}--\eqref{ew}. Without any further specification, we shall always assume that any initial data satisfies $g_0 = (u_0, v_0, w_0) \in H_+$. By Lemma \ref{L:locs}, there is a maximal interval of existence, denoted by $I_{max} = [0, \tau_{max}), \tau_{max} > 0$, for the corresponding weak solution $g(t) = g(t; g_0)$ such that for any $[0, \tau] \subset I_{max}$, 
\beq \bl{psoln}
	g \in C([0, \tau]; H_+) \cap C^1 ((0, \tau); H_+) \cap L^2 (0, \tau; E_+).
\eeq
By rescaling $w(t, x)$ and letting
\beq \bl{rsw}
	W (t, x) = \frac{c_2}{b_2} w (t, x),
\eeq
the system \eqref{eu}--\eqref{ew} becomes
\begin{align} 
	\frac{\partial u}{\partial t} &= d_1 \gd u + a_1 u + b_1 v - Fu^2 - G_1 uv,  \bl{eu1} \\
	\frac{\partial v}{\partial t} &= d_2 \gd v - b_2 v + b_2 W - G_2 uv,  \bl{ev1} \\
	\frac{b_2}{c_2} \frac{\partial W}{\partial t} &= \frac{d_3 \, b_2}{c_2} \gd W + a_3 u - \frac{c_3 \, b_2}{c_2} W. \bl{ew1}
\end{align}
Take the inner-products $\inpt{\eqref{eu1}, u(t, x)}, \inpt{\eqref{ev1}, v(t, x)}, \inpt{\eqref{ew1}, c_2W(t, x)/c_3}$ and add up the resulting equalities to obtain
\beq \bl{uvw}
	\begin{split}
	&\frac{1}{2} \frac{d}{dt} \left(\| u \|^2 + \|v \|^2 + \frac{b_2}{c_3} \|W \|^2 \right) + \left(d_1 \|\nb u\|^2 + d_2 \|\nb v\|^2 + \frac{d_3 \, b_2}{c_3} \|\nb W\|^2 \right) \\
	= &\, \int_\gw \left(a_1 u^2 + b_1 u v + \frac{a_3 \, c_2}{c_3} u W \right)\, dx - \int_\gw \left(F u^3 + G_1 u^2 v + G_2 u v^2 \right)\, dx \\
	&\, - b_2 \left(\| v \|^2 - \int_\gw v W \, dx + \| W \|^2 \right) \\
	\leq &\, \int_\gw \left(a_1 u^2 + b_1 u v + \frac{a_3 \, c_2}{c_3} u W \right)\, dx - \int_\gw F u^3 \, dx - \frac{b_2}{2} (\| v \|^2 + \|W \|^2) \\
	\leq &\, \int_\gw (M_1 u^2 - F u^3 )\, dx = \int_\gw \left(F^{-2/3} M_1 F^{2/3} u^2 - F u^3 \right)\, dx \\
	\leq &\, \frac{2}{3} \int_\gw F u^3 \, dx + \frac{M_1^3}{3F^2} | \gw | - \int_\gw F u^3 \, dx \leq  \frac{M_1^3}{3F^2} | \gw |,
	\end{split}
\eeq
where
$$
	M_1 = a_1 + \frac{1}{2 b_2} \left(b_1^2 + \left(\frac{a_3 \, c_2}{c_3}\right)^2 \right),
$$
and Young's inequality is used. Therefore, by \eqref{pcr} and letting 
\begin {equation} \label{dm}
	d_0 = \min \{d_1, d_2, d_3\}, \quad M_2 = \frac{c_2^2}{b_2\, c_3},
\end{equation}
we get
\begin{align*}
	&\frac{d}{dt} \left(\|u(t)\|^2 + \|v(t)\|^2 + \frac{b_2}{c_3} \|W(t)\|^2\right) \\
	+ &\, 2\ga d_0 \left(\|u(t)\|^2 + \|v(t)\|^2 + \frac{b_2}{c_3} \|W(t)\|^2 \right) \leq \frac{2M_1^3}{3F^2} |\gw|, \quad \tup{for} \; t \in I_{max}.
\end{align*}
By \eqref{rsw} and the Gronwall inequality it yields that for $t \in I_{max}$,

\beq \bl{L2b}
	\begin{split}
	\min \{1, M_2\} \|g(t; g_0)\|^2 \leq &\,\|u(t)\|^2 + \|v(t)\|^2 + M_2 \|w(t)\|^2 \\[3pt]
	= &\, \|u(t)\|^2 + \|v(t)\|^2 + \frac{b_2}{c_3} \|W(t)\|^2 \\
	\leq &\, e^{- 2\ga d_0 t} \left(\|u_0\|^2 + \|v_0\|^2 + \frac{b_2}{c_3} \|W(0)\|^2 \right) + \frac{M_1^3}{3\ga d_0 F^2} |\gw|\\
	= &\, e^{- 2\ga d_0 t} \left(\|u_0\|^2 + \|v_0\|^2 + M_2 \|w_0\|^2\right) + \frac{M_1^3}{3\ga d_0 F^2} |\gw|.
	\end{split}
\eeq
This shows that the weak solution $g(t; g_0)$ never blows up at any finite time and $I_{max} = [0, \infty)$ for every initial datum. Moreover, it follows that 
\beq  \bl{L2ab}
	\limsup_{t \to \infty} \|g(t; g_0)\|^2 \leq \frac{M_1^3}{3 \ga d_0 F^2 \min \{1, M_2\}} |\gw|.
\eeq
Therefore, \eqref{absb} hols with 
\beq \bl{k1}
	K_1 = \frac{M_1^3}{\ga d_0 F^2 \min \{1, M_2\}} |\gw|.
\eeq
The proof is completed.
\end{proof}

The solution semiflow $\csg$ of the Oregonator evolutionary equation \eqref{eveq} on $H_+$ will be briefly called the \emph{Oregonator semiflow}. In the next lemma, we show that the Oregonator semiflow has the absorbing property further in the product Banach spaces $[L^{2p} (\gw)]^3$ for $1 \leq p \leq 3$.

\begin{lemma} \bl{L:pab}
	For any integer $p, 1 \leq p \leq 3$, there exists a positive constant $K_p$ such that the Oregonator semiflow $\csg$ satisfies the absorbing inequality
\beq  \bl{psup}
	\limsup_{t \to \infty} \|S(t) (u_0, v_0, w_0)\|_{L^{2p}}^{2p} < K_p,
\eeq
for any initial datum $g_0 = (u_0, v_0, w_0) \in H_+$. Therefore, the Oregonator smiflow $\csg$ has the absorbing property in the space $[L^{2p} (\gw)]^3, 1 
\leq p \leq 3$.
\end{lemma}

\begin{proof}
	The case for $p = 1$ has been proved in Lemma \ref{L:glsn}. Now we show the lemma for $p = 3$. It then implies that the lemma also holds for $p = 2$.
	
	According to the property \eqref{soln} of the weak solution of the evolutionary equation \eqref{eveq}, here $T_{max} = \infty$ for all solutions, we see that for any given initial status $g_0 = (u_0, v_0, w_0) \in H_+$ there exists a time $t_0 \in (0, 1)$ such that 
$$
	S(t_0) g_0 \in E_+ = [H_0^1 (\gw)]^3 \subset [L^6 (\gw)]^3.
$$
Then the regularity of solutions of parabolic evolutionary equations shown in \cite[theorem 47.6]{SY02} ensures that
$$
	S(\cdot) g_0 \in C([t_0, \infty), E_+) \subset C([t_0, \infty), [L^6 (\gw)]^3),
$$
for space dimension $n \leq 3$. Based on this observation, without loss of generality, in considering the longtime behavior of the solution trajectories we can assume that $g_0 = (u_0, v_0, w_0) \in E_+ \subset [L^6 (\gw)]^3$. Moreover, by the bootstrap argument using the regularity of strong solutions of the parabolic evolutionary equation \eqref{eveq}, we have 
$$
	S(t) g_0 \in D(A) \subset [L^8 (\gw)]^3, \quad \textup{for} \; \, t > 0.
$$

Taking the $L^2$ inner-product $\inpt{\eqref{eu1}, u^5 (t, \cdot)}$, for $t > 0$, we get
\begin{align*}
	&\frac{1}{6} \frac{d}{dt} \int_\gw u^6 (t, x) \, dx + 5d_1 \|u^2 (t, \cdot) \nb u(t, \cdot)\|^2 \\
	= &\, a_1 \int_\gw u^6 (t, x) \, dx + b_1 \int_\gw u^5 (t, x) v(t, x)\, dx - F \int_\gw u^7 (t, x)\, dx - G_1 \int_\gw u^6 (t, x) v(t, x)\, dx.
\end{align*} 
Taking the $L^2$ inner-product $\inpt{\eqref{ev1}, v^5 (t, \cdot)}$, for $t > 0$, we get
\begin{align*}
	&\frac{1}{6} \frac{d}{dt} \int_\gw v^6 (t, x) \, dx + 5d_2 \|v^2 (t, \cdot) \nb v(t, \cdot)\|^2 \\
	= &\, - b_2 \int_\gw v^6 (t, x) \, dx + b_2 \int_\gw v^5 (t, x) W(t, x)\, dx - G_2 \int_\gw u (t, x) v^6 (t, x)\, dx.
\end{align*} 
Taking the $L^2$ inner-product $\inpt{\eqref{ew1}, c_2 W^5 (t, \cdot)/c_3}$, for $t > 0$, we get
\begin{align*}
	&\frac{b_2}{6\, c_3} \frac{d}{dt} \int_\gw W^6 (t, x) \, dx + \frac{5\, d_3 b_2}{c_3} \|W^2 (t, \cdot) \nb W(t, \cdot)\|^2 \\
	=  &\, \frac{a_3 \, c_2}{c_3} \int_\gw u (t, x) W^5 (t, x) \, dx - b_2 \int_\gw W^6 (t, x)\, dx. 
\end{align*}
Add up the above three equalities to obtain
\beq \bl{6es}
	\begin{split}
	& \frac{1}{6} \frac{d}{dt} \left(\|u(t, \cdot)\|_{L^6}^6 + \|v(t, \cdot)\|_{L^6}^6 + \frac{b_2}{c_3} \|W(t, \cdot)\|_{L^6}^6 \right)  \\
	&{}+ 5 \left(d_1 \|u^2 \nb u \|^2 + d_2 \|v^2 \nb v\|^2 + \frac{d_3 b_2}{c_3} \|W^2 \nb W\|^2\right) \\
	\leq &\,  a_1 \int_\gw u^6 \, dx + b_1 \int_\gw u^5 v \, dx - F \int_\gw u^7 (t, x)\, dx + \frac{a_3\, c_2}{c_3} \int_\gw u \, W^5 \, dx \\
	&{} - b_2 \int_\gw (v^6 - v^5\, W + W^6 )\, dx,
	\end{split}
\eeq
where, by Young's inequality, 
$$
	- b_2 \int_\gw (v^6 - v^5\, W + W^6 )\, dx \leq - b_2 \int_\gw (v^6 - \frac{5}{6} v^6 - \frac{1}{6} W^6 + W^6 )\, dx = - \frac{b_2}{6} \int_\gw (v^6 + 5W^6 )\, dx,
$$
and
\begin{align*}
	 b_1 \int_\gw u^5 v \, dx &\leq \frac{b_2}{6} \int_\gw v^6 \, dx + \frac{5\, b_1^{6/5}}{6\, b_2^{1/5}} \int_\gw u^6 \, dx, \\
	 \frac{a_3\, c_2}{c_3} \int_\gw u \, W^5 \, dx &\leq \frac{5\, b_2}{6} \int_\gw W^6 \, dx + \frac{1}{6\, b_2^5} \left(\frac{a_3\, c_2}{c_3}\right)^6 \int_\gw u^6 \, dx.
\end{align*}
Substituting the above three inequalities into \eqref{6es}, we have
\beq \bl{6ES}
	\begin{split}
	 &\frac{1}{6} \frac{d}{dt} \left(\|u\|_{L^6}^6 + \|v\|_{L^6}^6 + \frac{b_2}{c_3} \|W\|_{L^6}^6 \right) + \frac{5}{3} \left(d_1 \| \nb (u^3) \|^2 + d_2 \| \nb (v^3)\|^2 + \frac{d_3\, b_2}{c_3} \| \nb (W^3) \|^2\right) \\
	\leq &\, \int_\gw \left(M_3 u^6 - F u^7 \right)\, dx \leq \frac{M_3^7}{7F^6} |\gw| + \frac{6}{7} \int_\gw F u^7 \, dx - F \int_\gw u^7\, dx \leq \frac{M_3^7}{7F^6} |\gw| ,
	\end{split}
\eeq
where
$$
	M_3 = a_1 +  \frac{5\, b_1^{6/5}}{6\, b_2^{1/5}} + \frac{1}{6\, b_2^5} \left(\frac{a_3\, c_2}{c_3}\right)^6.
$$
It follows that, by \eqref{pcr} and \eqref{dm},
\begin{align*}
	&\frac{d}{dt} \left(\|u\|_{L^6}^6 + \|v\|_{L^6}^6 + \frac{b_2}{c_3} \|W\|_{L^6}^6 \right) + 10\, \ga \, d_0 \left(\|u\|_{L^6}^6 + \|v\|_{L^6}^6 + \frac{b_2}{c_3} \|W\|_{L^6}^6 \right) \\[3pt]
	\leq &\, \frac{d}{dt} \left(\|u\|_{L^6}^6 + \|v\|_{L^6}^6 + \frac{b_2}{c_3} \|W\|_{L^6}^6 \right) + 10 \left(d_1 \| \nb u^3 \|^2 + d_2 \| \nb v^3\|^2 + \frac{d_3\, b_2}{c_3} \| \nb W^3 \|^2\right) \\[3pt]
	\leq &\, \frac{6\, M_3^7}{7\, F^6} |\gw| <  \frac{M_3^7}{F^6} |\gw| , \quad t > 0.
\end{align*}
By \eqref{rsw}, we have
$$
	\frac{b_2}{c_3} \| W \|_{L^6}^6 = M_4 \| w \|_{L^6}^6, \quad \tup{where} \; M_4 = \frac{c_2^6}{b_2^5 \, c_3},
$$
and, by the Gronwall inequality, we end up with
\beq \bl{6bd}
	\begin{split}
	&\| (u(t, \cdot), v(t, \cdot), w(t, \cdot))\|_{L^6}^6  \\[5pt]
	\leq &\, \frac{\max \{1, M_4\}}{\min \{1,  M_4\}} e^{-10 \ga d_0 t} \|(u_0, v_0, w_0)\|_{L^6}^6 + \frac{M_3^7}{10 \gamma d_0 F^6 \min \{1, M_4\}} |\gw|,  \quad t \geq 0,
	\end{split}
\eeq
so that
\beq \bl{6lim}
	\limsup_{t \to \infty} \|S(t) (u_0, v_0, w_0) \|_{L^6}^6 < K_3,
\eeq
where
$$
	  K_3 =  \frac{M_3^7}{\gamma d_0 F^6 \min \{1, M_4\}} |\gw|.
$$

Similarly, or as a consequence of \eqref{6lim}, one can show that there exists a positive constant $K_2 > 0$ such that
\beq \bl{4lim}
	\limsup_{t \to \infty} \|S(t) (u_0, v_0, w_0) \|_{L^4}^4 < K_2.
\eeq
Thus the proof is completed.
\end{proof}

\section{\textbf{Asymptotic Compactness and Global Attractor}}

In this section we shall prove that the Oregonator semiflow is asymptotically compact in $H$ and has a global attractor.

\begin{lemma} \label{L:ac}
	The Oregonator semiflow $\csg$ is asymptotically compact in the invariant cone $H_+$ of the phase space $H$.
\end{lemma}
\begin{proof}
Taking the $L^2$ inner-product $\inpt{\eqref{eu1}, -\gd u}$, by the homogeneous Dirichlet boundary conditions, we get
\beq \bl{dtu}
	\begin{split}
	&\frac{1}{2} \frac{d}{dt} \|\nb u \|^2 + d_1 \|\gd u \|^2 \\
	= &\, - \int_\gw (a_1 u + b_1v) \gd u \, dx - 2F \int_\gw u |\nb u |^2 \, dx + G_1 \int_\gw uv \gd u \, dx \\
	\leq &\, \left(\frac{d_1}{4} + \frac{d_1}{4} + \frac{d_1}{2}\right) \|\gd u\|^2 + \frac{1}{d_1} \int_\gw \left(a_1^2 u^2 + b_1^2 v^2 + \frac{G_1^2}{2} u^2 v^2 \right) dx  \\
	\leq &\, d_1 \|\gd u\|^2 + \frac{1}{d_1} (a_1^2 + b_1^2)(\|u\|^2 + \|v\|^2 ) + \frac{G_1^2}{4 d_1} \left(\| u \|_{L^4}^4 + \| v \|_{L^4}^4 \right), \quad t > 0.
	\end{split}
\eeq
Taking the $L^2$ inner-product $\inpt{\eqref{ev1}, -\gd v}$, similarly we get
\beq \bl{dtv}
	\begin{split}
	&\frac{1}{2} \frac{d}{dt} \|\nb v \|^2 + d_2 \|\gd v \|^2 \\
	= &\, - b_2 \int_\gw (|\nb v |^2 - \nb v \cdot \nb W)\, dx  + G_2 \int_\gw uv \gd v \, dx \\
	\leq &\, - b_2 \int_\gw (|\nb v |^2 - \nb v \cdot \nb W)\, dx  +  \frac{d_2}{2} \|\gd v \|^2 + \frac{G_2^2}{2 d_2} \int_\gw u^2 v^2 \, dx \\
	\leq &\, - b_2 \int_\gw (|\nb v |^2 - \nb v \cdot \nb W)\, dx  +  \frac{d_2}{2} \|\gd v \|^2 + \frac{G_2^2}{4 d_2} \left(\| u \|_{L^4}^4 + \| v \|_{L^4}^4 \right), \quad t > 0.
	\end{split}
\eeq
Taking the $L^2$ inner-product $\inpt{\eqref{ew1}, - c_2\gd W /c_3}$, we have
\beq \bl{dtw}
	\begin{split}
	&\frac{b_2}{2 c_3} \frac{d}{dt} \|\nb W \|^2 + \frac{d_3 \, b_2}{c_3} \|\gd W \|^2 \\
	= &\, - \frac{a_3 \, c_2}{c_3} \int_\gw u \, \gd W \, dx - b_2 \int_\gw |\nb W |^2 \, dx \\
	\leq &\, \frac{c_2}{c_3} \left(\frac{d_3 b_2}{c_2} \| \gd W \|^2 + \frac{c_2}{4 d_3 b_2} \int_\gw a_3^2 u^2 \, dx \right)  - b_2 \int_\gw |\nb W |^2 \, dx \\
	= &\,  \frac{d_3 \, b_2}{c_3} \|\gd W \|^2 + \frac{a_3^2 \, c_2^2}{4 d_3\, b_2\, c_3} \int_\gw u^2 \, dx - b_2 \int_\gw |\nb W |^2 \, dx, \quad t > 0.
	\end{split}
\eeq
Since $\|\nb W \|^2 = (c_2 / b_2)^2 \|\nb w \|^2$ and
$$
	- b_2 \int_\gw (|\nb v |^2 - \nb v \cdot \nb W + |\nb W |^2)\, dx \leq 0,
$$
summing up \eqref{dtu}, \eqref{dtv} and \eqref{dtw} we obtain the following inequality
\beq \bl{dts}
	\begin{split}
	&\frac{d}{dt} \left(\|\nb u \|^2 + \|\nb v \|^2 + \frac{c_2^2}{b_2 \, c_3} \|\nb w \|^2\right) \\
	\leq &\, \frac{2}{d_1} (a_1^2 + b_1^2)(\|u\|^2 + \|v\|^2 )  + \frac{a_3^2 \, c_2^2}{2 d_3\, b_2\, c_3} \| u \|^2 + \left(\frac{G_1^2}{2d_1} + \frac{G_2^2}{2d_2}\right)\left(\| u \|_{L^4}^4 + \| v \|_{L^4}^4 \right), 
	\end{split}
\eeq
for $t > 0$. Note that we have taken $\|\nb \varphi \|$ as the norm of $E$ and there is a positive constant $\eta > 0$ associated with the Sobolev imbedding inequality
\beq \bl{4imb}
	\| \varphi \|_{L^4 (\gw)} \leq \eta \| \varphi \|_E = \eta \|\nb \varphi \|, \quad \textup{for any} \; \varphi \in E.
\eeq
Since $B_0$ in \eqref{absb} is an absorbing ball, there is a finite time $T_0 > 0$ depending only on $B_0$ such that $S(t) B_0 \subset B_0$ for all $t > T_0$. Moreover, from \eqref{L2b} with $t \in [0, \infty)$ and \eqref{k1} we can assert that there exists a finite time $T_1 > 0$ depending only on $B_0$ such that 
\begin{equation} \label{tk1}
	\| u(t) \|^2 + \|v(t) \|^2 + \|w(t)\|^2 = \|g(t; g_0)\|^2 \leq K_1, \quad \textup{for any} \; t > T_1, \; g_0 \in B_0.
\end{equation}
Then \eqref{dts} along with these facts shows that for any initial datum $g_0 \in B_0$ one has
\beq \bl{ugr}
	\begin{split}
	&\frac{d}{dt} \|(\nb u, \nb v, \sqrt{M_2} \nb w)\|^2 \\
	\leq &\, \left(\frac{G_1^2}{2d_1} + \frac{G_2^2}{2d_2}\right) \eta^4 (\|\nb u \|^4 + \|\nb v \|^4) + \left(\frac{2}{d_1} (a_1^2 + b_1^2) + \frac{a_3^2 \, c_2^2}{2 d_3\, b_2\, c_3}\right) K_1 \\
	 \leq &\, \left(\frac{G_1^2}{2d_1} + \frac{G_2^2}{2d_2}\right) \eta^4 \| (\nb u, \nb v, \sqrt{M_2} \nb w) \|^4 + \left(\frac{2}{d_1} (a_1^2 + b_1^2) + \frac{a_3^2 \, c_2^2}{2 d_3\, b_2\, c_3}\right) K_1, \;\, t > T_0 + T_1, 
	\end{split}
\eeq
where $M_2$ is shown in \eqref{dm}. The differential inequality \eqref{ugr} can be written as 
\beq \bl{Grw}
	\frac{d}{dt} \beta \leq \rho \, \beta + h, \quad \tup{for} \; t > T_0 + T_1, \; g_0 \in B_0,
\eeq
where
$$
	\beta (t) = \| (\nb u, \nb v, \sqrt{M_2} \nb w) \|^2, \quad \rho (t) = \left(\frac{G_1^2}{2d_1} + \frac{G_2^2}{2d_2}\right) \eta^4 \beta (t), 
$$
and
$$
	 h(t) = K_1 \left(\frac{2}{d_1} (a_1^2 + b_1^2) + \frac{a_3^2 \, c_2^2}{2 d_3\, b_2\, c_3}\right).
$$
From \eqref{uvw}, \eqref{dm} and \eqref{tk1} we see that, for any given initial status $g_0 = (u_0, v_0, w_0) \in B_0$, 
\beq \bl{ttb}
	\begin{split}
	\int_t^{t +1} \beta (s) \, ds &\leq \frac{1}{d_0} \left(\max \{1, M_2\}\|(u(t), v(t), w(t))\|^2 + \frac{M_1^3}{F^2} |\gw|\right) \\
	& \leq \frac{1}{d_0} \left(K_1 \max \{1, M_2\} +  \frac{M_1^3}{F^2} |\gw| \right), \; \textup{for} \; t > T_0 + T_1, \; g_0 \in B_0.
	\end{split}
\eeq
Let
\begin{equation} \label{m5}
	M_5 = \frac{1}{d_0} \left(K_1 \max \{1, M_2\} +  \frac{M_1^3}{F^2} |\gw| \right).
\end{equation}
Then we can apply the uniform Gronwall inequality, cf. \cite{rT88, SY02}, to \eqref{Grw} and use \eqref{ttb} and \eqref{m5} to get
\beq \bl{ascc}
	\begin{split}
	&\|(\nb u (t, \cdot), \nb v(t, \cdot), \nb w (t, \cdot))\|^2 \leq \frac{1}{\min \{1, M_2\}} \, \beta (t) \\[2pt]
	\leq \frac{1}{\min \{1, M_2\}} &\left(M_5 +  K_1 \left(\frac{2}{d_1} (a_1^2 + b_1^2) + \frac{a_3^2 \, c_2^2}{2 d_3\, b_2\, c_3}\right)\right)\exp \left(\eta^4 M_5\left(\frac{G_1^2}{2d_1} + \frac{G_2^2}{2d_2}\right) \right), 
	\end{split}
\eeq
for any $t > T_0 + T_1 + 1,  g_0 \in B_0$.

The boundedness shown by \eqref{ascc} combined with the absorbing property shown in Lemma \ref{L:glsn} confirms that, for any given bounded set $B \subset H$, there exists a finite time $T(B) > 0$ such that $\{S(t) B: t > T(B)\}$ is a bounded set in $E$, which in turn is a precompact set in $H$ due to that $E$ is compactly imbedded in $H$. Therefore, the Oregonator semiflow $\csg$ is asymptotically compact in $H_+$.
\end{proof}

Finally we can prove the main result on the existence of a global attractor for $\csg$. 
\begin{theorem} \label{T:main} 
Given any positive parameters in the Oregonator system \eqref{eu}--\eqref{ew} with the Dirichlet boundary condition \eqref{dbc}, there exists a global attractor $\ms{A}$ in $H_+$ for the Oregonator semiflow $\csg$ generated by \eqref{eveq}.
\end{theorem}
\begin{proof}
Lemma \ref{L:glsn} and Lemma \ref{L:ac} demonstrate that the two conditions in Proposition \ref{P:exga} are satisfied by the Oregonator semiflow $\csg$, where we let $\mX = H_+$. Therefore, we reach the conclusion.
\end{proof}
We emphasize that here the existence of global attractor in $H_+$ is established unconditionally for any given positive parameters involved in this Oregonator system.

\section{\textbf{Finite Dimensionality of the Global Attractor}}

Consider the Hausdorff dimension and fractal dimension of the global attractor $\ms{A}$ of the Oregonator semiflow $\csg$ in $H_+$. Let $q_{m} = \limsup_{t \to \infty} \, q_{m} (t)$, where, cf. \cite{rT88},
\begin{equation} \label{trq} 
	q_{m} (t) = \sup_{g_0 \in \ms{A}} \;  \, \sup_{\substack{g_{i} \in H, \|g_{i} \| = 1\\ i = 1, \cdots, m}} \; \, \left( \frac{1}{t} \int_{0}^{t} \textup{Tr}  \left( A + f^{\prime} (S(\tau) g_0 ) \right) \circ Q_{m} (\tau) \, d\tau \right),
\end{equation}
in which  $Q_{m} (t)$ stands for the orthogonal projection of space $H$ on the subspace spanned by $G_1 (t), \cdots, G_{m} (t)$, with $G_{i} (t) = L(S(t), g_0)g_{i}, i = 1, \cdots, m$. Here $f^{\prime}(S(\tau)g_0)$ is the Fr\'{e}chet derivative of the map $f$ at  $S(\tau)g_0$, and $L(S(t), g_0)$ is the Fr\'{e}chet derivative of the map $S(t)$ at $g_0$, with $t$ fixed. The definitions of Hausdorff dimension and fractal dimension can be seen in \cite[Chapter 5]{rT88} as well as the following proposition.

\begin{proposition} \label{P:HFd}
If there is an integer $m$ such that $q_{m} < 0$, then the Hausdorff dimension $d_{H} (\ms{A})$ and the fractal dimension $d_{F} (\ms{A})$ of $\ms{A}$ satisfy
\begin{equation}  \label{hfd}
	d_{H} (\ms{A}) \leq m,  \quad \textup{and} \quad d_{F} (\ms{A}) \leq m \max_{1 \leq j \leq m - 1} \left( 1 + \frac{(q_{j})_{+}}{| q_{m} |} \right) \leq 2m. 
\end{equation}
\end{proposition} 
It is standard to show that for any given $t > 0$, $S(t)$ on $H_+$ is Fr\'{e}chet differentiable and and its Fr\'{e}chet derivative at $g_0$ is given by
$$
	L(S(t), g_0)Z_0 \overset{\textup{def}}{=} Z(t) = (\mU(t), \mV(t), \mW(t)), 
$$
for any $Z_0 = (\mU_0, \mV_0, \mW_0) \in H$, where $(\mU(t), \mV(t), \mW(t))$ is the weak solution of the following initial-boundary value problem of the variational system associated with the trajectory $\{S(t) g_0: t \geq 0\}$,
\begin{equation} \label{vareq}
	\begin{split}
	\frac{\partial \mU}{\partial t} & = d_1 \gd \mU + a_1 \mU + b_1 \mV -2F u(t) \mU - G_1v(t) \mU - G_1 u(t) \mV,  \\
	\frac{\partial \mV}{\partial t} & = d_2 \gd \mV - b_2 \mV + c_2 \mW - G_2 v(t) \mU - G_2 u(t) \mV, \\
	\frac{\partial \mW}{\partial t} & = d_1 \gd \mW + a_3 \mU - c_3 \mW,  \quad t > 0, \; x \in \gw, \\
	 & \mU\mid_{\partial \gw} =  \mV\mid_{\partial \gw} = \mW\mid_{\partial \gw} = 0, \quad t > 0, \\[2pt]
	 & \mU(0) = \mU_0, \quad \mV(0) = \mV_0, \quad \mW(0) = \mW_0. 
	\end{split}
\end{equation}
Here $(u(t), v(t), w(t)) = g(t) = S(t)g_0$ is the weak solution of \eqref{eveq} satisfying the initial condition $g(0) = g_0$.  The initial-boundary value problem \eqref{vareq} can be written as 
\begin{equation} \label{vareveq}
	\begin{split}
	\frac{dZ}{dt} = (A + & f^{\prime} (S(t)g_0)) Z,   \quad t > 0, \\[2pt]
	&Z(0) = Z_0.
	\end{split}
\end{equation}
\begin{theorem} \label{T:fnd}
The global attractors $\ms{A}$ for the Oregonator semiflow $\csg$ has a finite Hausdorff dimesion and a finite fractal dimension.
\end{theorem}
\begin{proof}
By Proposition \ref{P:HFd}, we shall estimate $\textup{Tr} \, (A + f^{\prime} (S(\tau)g_0 )) \circ Q_{m}(\tau)$. At any given time $\tau > 0$, let $\{\varphi_{j} (\tau): j = 1, \cdots , m\}$ be an $H$-orthonormal basis for the subspace $Q_m(\tau) H = \textup{Span}\,  \{Z_1(\tau), \cdots , Z_,(\tau) \}$, where $Z_1 (t), \cdots , Z_{m} (t)$ are the weak solutions of \eqref{vareveq} with the respective initial data $Z_{1,0}, \cdots , Z_{m,0}$ and, without loss of generality, assuming that $Z_{1,0}, \cdots , Z_{m,0}$ are linearly independent in $H$. 

Note that $Z_1 (t), \cdots , Z_{m} (t)$ turn out to be strong solutions for $t > 0$. By Gram-Schmidt orthogonalization,  $\varphi_{j}(\tau) = (\varphi_{j}^1 (\tau), \varphi_{j}^2 (\tau),\varphi_{j}^3 (\tau)) \in E$ for $\tau > 0$, $j = 1, \cdots , m$,  and $\varphi_{j} (\tau)$ are strongly measurable in $\tau$. Recall that $d_0 = \min \{d_1, d_2, d_3 \}$. Then 
\begin{equation} \label{Trace}
		\begin{split}
	\textup{Tr} \, (A + f^{\prime} (S(\tau)g_0 )\circ Q_m(\tau)& = \sum_{j=1}^{m} \left( \langle A \varphi_{j}(\tau), \varphi_{j}(\tau) \rangle + \langle  f^{\prime} (S(\tau)g_0 ) \varphi_{j}(\tau), \varphi_{j}(\tau) \rangle\right)  \\
	& \leq - d_0 \sum_{j=1}^{m} \, \| \nabla \varphi_{j}(\tau) \|^2 + J_1 + J_2 + J_3,  \quad \tau > 0,
		\end{split}
\end{equation}
where the three terms $J_1, J_2, J_3$ are given by

\begin{align*}
	J_1 &= - \sum_{j=1}^{m}  \int_{\gw} (2F u(\tau) + G_1 v(\tau)) |\varphi_{j}^1 (\tau) |^2\, dx - \sum_{j=1}^{m} \, \int_{\gw} G_1 u(t)  \varphi_{j}^1 (\tau)  \varphi_{j}^2 (\tau) \,dx \\
	&\leq - \sum_{j=1}^{m} \, \int_{\gw} G_1 u(t)  \varphi_{j}^1 (\tau)  \varphi_{j}^2 (\tau) \,dx, \\
	J_2 &= - \sum_{j=1}^{m} \int_{\gw}  G_2 v(\tau) \varphi_{j}^1 (\tau) \varphi_{j}^2 (\tau) \, dx - \sum_{j=1}^{m} \int_{\gw} G_2 u(\tau) |\varphi_{j}^2 (\tau)|^2 \, dx \\
	&\leq  - \sum_{j=1}^{m} \int_{\gw}  G_2 v(\tau) \varphi_{j}^1 (\tau) \varphi_{j}^2 (\tau) \, dx,
\end{align*}
and
\begin{align*}
	J_3  =&\,  \sum_{j=1}^{m} \int_{\gw} \left( a_1 |\varphi_{j}^1 (\tau) |^2 + b_1 \varphi_{j}^1 (\tau) \varphi_{j}^2 (\tau) - b_2 |\varphi_{j}^2 (\tau)|^2 \right) dx \\
	&{} + \sum_{j=1}^{m} \int_{\gw} \left( c_2 \varphi_j^2 (\tau) \varphi_j^3 (\tau) + a_3 \varphi_j^1 (\tau) \varphi_j^3 (\tau) - c_3 |\varphi_j^3 (\tau)|^2 \right) dx \\
	 \leq &\, \sum_{j=1}^{m} \int_{\gw} \left(a_1 |\varphi_{j}^1 (\tau) |^2 + b_1 \varphi_{j}^1 (\tau) \varphi_{j}^2 (\tau) +  c_2 \varphi_j^2 (\tau) \varphi_j^3 (\tau) + a_3 \varphi_j^1 (\tau) \varphi_j^3 (\tau)\right) dx.
\end{align*}
By the generalized H\"{o}lder inequality, we get
\begin{equation} \label{J1eq}
		\begin{split}
	J_1 &\leq G_1 \sum_{j=1}^{m} \| u(\tau) \| \| \varphi_{j}^1 (\tau) \|_{L^4} \| \varphi_{j}^2 (\tau) \|_{L^4} \\
	&\leq G_1 \sum_{j=1}^{m} \|S(\tau) g_0 \| \| \varphi_{j}^1 (\tau) \|_{L^4} \| \varphi_{j}^2 (\tau) \|_{L^4} \leq G_1 \sqrt{K_1} \sum_{j=1}^{m} \| \varphi_j (\tau) \|_{L^4}^2, 
		\end{split}
\end{equation}
for any $\tau > 0$ and any $g_0 \in \ms{A}$. Now we apply the Garliardo-Nirenberg interpolation inequality, cf. \cite[Theorem B.3]{SY02},
\begin{equation} \label{GNineq}
	\| \varphi \|_{W^{k,p}} \leq C \| \varphi \|_{W^{m,q}}^{\theta} \| \varphi \|_{L^{r}}^{1 - \theta}, \quad \textup{for} \; \varphi \in W^{m,q}(\gw),
\end{equation}
provided that $p, q, r \geq 1, 0 < \theta < 1$, and
$$
	k - \frac{n}{p} \leq \theta \left( m - \frac{n}{q} \right)  - (1 - \theta ) \frac{n}{r} ,   \quad \textup{where} \; \, n = \textup{dim} \, \gw.
$$
Here let $W^{k, p}(\gw) = L^4(\gw), W^{m, q}(\gw) = H_{0}^{1}(\gw), L^{r}(\gw) = L^2(\gw)$, and $\theta = n/4 \leq 3/4$. It follows from \eqref{GNineq} that
\begin{equation} \label{inter}
	\| \varphi_{j} (\tau) \|_{L^4} \leq C \| \nabla \varphi_{j} (\tau) \|^{\frac{n}{4}} \| \varphi_{j} (\tau) \|^{1 - \frac{n}{4}} = C \| \nabla \varphi_{j} (\tau) \|^{\frac{n}{4}}, \quad j = 1, \cdots , m,
\end{equation}
since $\| \varphi_{j}(\tau) \| = 1$, where $C$ is a uniform constant. Substitute \eqref{inter} into \eqref{J1eq} to obtain
\begin{equation*} 
	J_1 \leq G_1 \sqrt{K_1} C^2 \sum_{j=1}^{m} \, \| \nabla \varphi_{j} (\tau) \|^{\frac{n}{2}}.
\end{equation*}
Similarly we can get
\begin{equation*} 
	J_2 \leq G_2 \sqrt{K_1} \sum_{j=1}^{m} \| \varphi_{j} (\tau) \|_{L^4}^2 \leq G_2 \sqrt{K_1} C^2 \sum_{j=1}^{m} \, \| \nabla \varphi_{j} (\tau) \|^{\frac{n}{2}}.
\end{equation*}
Moreover, we have 
\begin{equation*} 
	J_3 \leq \sum_{j=1}^{m} (a_1 + b_1 + c_2 + a_3) \| \varphi_{j} (\tau) \|^2 = m (a_1 + b_1 + c_2 + a_3).
\end{equation*}
Substituting the above three inequalities into \eqref{Trace}, we obtain
\begin{equation} \label{Trest}
	\begin{split}
	&\textup{Tr} \, (A + f^{\prime} (S(\tau)g_0 )\circ Q_m(\tau) \\
	\leq &\, - d_0 \sum_{j=1}^{m} \| \nabla \varphi_{j}(\tau) \|^2 +  (G_1 + G_2) \sqrt{K_1} C^2 \sum_{j=1}^{m} \| \nabla \varphi_{j}(\tau) \|^{\frac{n}{2}} + m (a_1 + b_1 + c_2 + a_3).
	\end{split}
\end{equation}
By Young's inequality, for $n \leq 3$, we have
$$
	(G_1 + G_2) \sqrt{K_1} C^2 \sum_{j=1}^{m} \| \nabla \varphi_{j}(\tau) \|^{\frac{n}{2}}  \leq \frac{d_0}{2} \sum_{j=1}^{m} \|\nabla \varphi_{j}(\tau) \|^2 + K(n) m,
$$
where $K(n)$ is a positive constant depending only on $n = $ dim $\gw$ and the involved constants $d_0, C, K_1, G_1$ and $G_2$. Hence, for any $\tau > 0$ and any $g_0 \in \ms{A}$, it holds that
\begin{equation*} 
	\textup{Tr} \, (A + f^{\prime} (S(\tau)g_0 )\circ Q_m(\tau) \leq - \frac{d_0}{2} \sum_{j=1}^{m} \| \nabla \varphi_{j}(\tau) \|^2 + m\left(K(n) +  a_1 + b_1 + c_2 + a_3\right). 
\end{equation*}
According to the generalized Sobolev-Lieb-Thirring inequality \cite[Appendix, Corollary 4.1]{rT88}, since $\{ \varphi_1 (\tau), \cdots , \varphi_{m} (\tau) \}$ is an orthonormal set in $H$, there exists a constant $\Psi > 0$ only depending on the shape and dimension of $\gw$ such that
\begin{equation} \label{SLTineq}
	\sum_{j=1}^{m} \| \nabla \varphi_{j}(\tau) \|^2  \geq  \frac{\Psi \, m^{1 + \frac{2}{n}}}{|\gw |^{\frac{2}{n}}}.
\end{equation}
Therefore, for any $\tau > 0$ and any $g_0 \in \ms{A}$,
\begin{equation} \label{finest}
	\textup{Tr} \, (A + f^{\prime} (S(\tau)g_0 )\circ Q_m(\tau) \leq - \frac{d_0  \Psi}{2 |\gw |^{\frac{2}{n}}} m^{1 + \frac{2}{n}} + m\left(K(n) + a_1 + b_1 + c_2 + a_3 \right).
\end{equation}
Then we conclude that
\begin{equation} \label{qmt}
		\begin{split}
	q_{m}(t) & =  \sup_{g_0 \in \ms{A}} \;  \, \sup_{\substack{g_{i} \in H, \|g_{i} \| = 1\\ i = 1, \cdots, m}} \; \, \left( \frac{1}{t} \int_{0}^{t} \textup{Tr}  \left( A + f^{\prime} (S(\tau) g_0 ) \right) \circ Q_{m} (\tau) \, d\tau \right)  \\
	& \leq - \, \frac{d_0 \Psi}{2 |\gw |^{\frac{2}{n}}} m^{1 + \frac{2}{n}}  + m\left(K(n) + a_1 + b_1 + c_2 + a_3\right), \quad \textup{for any} \, \; t > 0, 
		\end{split}
\end{equation}
so that 
\begin{equation}  \label{qm}
	q_m = \limsup_{t \to \infty} \, q_m (t) \leq  - \, \frac{d_0 \Psi}{2 |\gw |^{\frac{2}{n}}} m^{1 + \frac{2}{n}} + m\left(K(n) + a_1 + b_1 + c_2 + a_3\right)  < 0,
\end{equation}
if the integer $m$ satisfies the following condition,
\begin{equation} \label{dimc}
	m - 1 \leq \left( \frac{2(K(n) + a_1 + b_1 + c_2 + a_3)}{d_0 \Psi} \right)^{n/2} | \gw | < m.
\end{equation}
According to Proposition \ref{P:HFd}, we have shown that the Hausdorff dimension and the fractal dimension of the global attractor $\ms{A}$ are finite with the upper bounds given by
$$
	d_{H} (\ms{A}) \leq m \quad \textup{and} \quad d_{F} (\ms{A}) \leq 2m,
$$
respectively, where the integer $m$ satisfies \eqref{dimc}. 
\end{proof}

\section{\textbf{$(H_+, E_+)$ Global Attractor and $L^\infty$ Regularity}}

In this section we show that the global attractor $\ms{A}$ of the Oregonator semiflow is an $(H_+, E_+)$ global attractor with the regularity $\ms{A} \subset [L^\infty (\gw)]^3$. The following concept was introduced in \cite{BV83}.

\begin{definition} \label{D:hea}
	Let $\mX$ be a Banach space or a closed invariant cone in a Banach space and $\{\Sigma (t)\}_{t\geq 0}$ be a semiflow on $\mX$. Let $\mY$ be a compactly imbedded subspace or sub-cone of $\mX$. A subset $\mathcal{A}$ of $\mY$ is called an $(\mX, \mY)$ global attractor for this semiflow if $\mathcal{A}$ has the following properties,
	
	\textup{(i)} $\mathcal{A}$ is a nonempty, compact, and invariant set in $\mY$.
	
	\textup{(ii)} $\mathcal{A}$ attracts any bounded set $B \subset \mX$ with respect to the $Y$-norm, namely, there is a time $\tau = \tau (B)$ such that $\Sigma (t)B \subset \mY$ for $t > \tau$ and $dist_\mY (\Sigma (t)B, \mathcal{A}) \to 0$, as $t \to \infty$.
\end{definition}

\begin{lemma} \label{L:ste}
	Let $\{g_m\}$ be a sequence in $E$ such that $\{g_m\}$ converges to $g_0 \in E$ weakly in $E$ and $\{g_m\}$ converges to $g_0$ strongly in $H$, as $m \to \infty$. Then
	$$
		\lim_{m \to \infty} S(t) g_m = S(t) g_0 \; \; \textup{strongly in} \; E,
	$$
where the convergence is uniform with respect to $t$ in any given compact interval $[t_0, t_1] \subset (0, \infty)$.
\end{lemma}

The proof of this lemma is seen in \cite[Lemma 4.2]{yY10}.

\begin{theorem} \label{T:HEatr}
	The global attractor $\ms{A}$ in $H_+$ for the Oregonator semiflow $\csg$ is indeed an $(H_+, E_+)$ global attractor.
\end{theorem}

\begin{proof}
By the proof of Lemma \ref{L:ac} and \eqref{ascc}, we find that
\beq  \bl{absbE}
	B_1 = \{\varphi \in E_+ : \|\varphi \|_E = \|\nb \varphi \|^2 \leq K_E \},
\eeq
where 
\begin{equation}  \bl{ke}
	K_E = \frac{1}{\min \{1, M_2\}} \left(M_5 +  K_1 \left(\frac{2}{d_1} (a_1^2 + b_1^2) + \frac{a_3^2 \, c_2^2}{2 d_3\, b_2\, c_3}\right)\right)\exp \left(\eta^4 M_5\left(\frac{G_1^2}{2d_1} + \frac{G_2^2}{2d_2}\right) \right), 
\end{equation}
is an absorbing set for the Oregonator semiflow $\csg$ in $E_+$. Indeed, for any $E$-bounded subset $B \subset E_+$, $B$ must also be bounded in $H_+$ so that there is a finite time $T^0 (B) \geq 0$ such that $S(t)B \subset B_0$ for all $t > T^0$. Then \eqref{ascc} implies that 
\begin{equation} \label{bb1}
	S(t) B \subset B_1, \quad \textup{for any} \; \, t > T^0 + T_0 + T_1 + 1,
\end{equation}
where $T_0$ and $T_1$ have been specified in the proof of Lemma \ref{L:ac}.

Next we show that the Oregonator semiflow $\csg$ is asymptotically compact with respect to the strong topology in $E$. For any time sequence $\{t_n \}, t_n \to \infty$, and any $E$-bounded sequence $\{g_n \} \subset E_+$, there exists a finite time $t_0 \geq 0$ such that $S(t) \{g_n\} \subset B_0$, for any $t > t_0$. Then for an arbitrarily given $T > t_0 + T_0 + T_1 + 1$, there is an integer $n_0 \geq 1$ such that $t_n > 2T$ for all $n > n_0$. 

By Lemma \ref{L:ac}, it holds that
$$
	\{S(t_n - T) g_n\}_{n > n_0} \; \textup{is a bounded set in}\; E_+.
$$	
Since $E$ is a Hilbert space, there is an increasing sequence of integers $\{n_j\}_{j=1}^\infty$, with $n_1 > n_0$, such that
$$
	  \lim_{j \to \infty} S(t_{n_j} - T) g_{n_j} = g^* \;\; \textup{weakly in} \; E.
$$
By the compact imbedding $E \hookrightarrow H$, there is a subsequence of $\{n_j\}$, which is relabeled as the same as $\{n_j\}$, such that
$$
	\lim_{j \to \infty} S(t_{n_j} - T) g_{n_j} = g^* \;\; \textup{strongly in} \; H_+,
$$
because $H_+$ is a closed invariant cone of $H$. Moreover, the uniqueness of limit implies that $g^* \in E_+$. Then by Lemma \ref{L:ste}, we have the following convergence with respect to the $E$-norm,
\beq  \bl{ace}
	\lim_{j \to \infty} S(t_{n_j}) g_{n_j} = \lim_{j \to \infty} S(T) S(t_{n_j} - T) g_{n_j} = S(T) g^* \;\; \textup{strongly in} \; E_+.
\eeq
This proves that $\csg$ is asymptotically compact on $E_+$. 

Therefore, by Proposition \ref{P:exga}, there exists a global attractor $\ms{A}_E$ for this Oregonator semiflow $\csg$ in the invariant cone $E_+$. Note that $B_1$ attracts the $H$-absorbing ball $B_0$ in the $E$-norm as demonstrated earlier in this proof, we see that this global attractor $\ms{A}_E$ is an $(H_+, E_+)$ global attractor according to Definition \ref{D:hea}. Thus the invariance and the boundedness of $\ms{A}$ in $H$ and of  $\ms{A}_E$ in $E$ imply that 
\begin{align*}
	& \ms{A} \; \textup{attracts} \; \ms{A}_E \; \textup{in} \; H_+, \; \textup{so that} \; \ms{A}_E \subset \ms{A}, \\
	&\ms{A}_E \; \textup{attracts} \; \ms{A} \; \textup{in} \; E_+, \; \textup{so that} \; \ms{A} \subset \ms{A}_E. 
\end{align*}
Therefore, $\ms{A} = \ms{A}_E$ and, as a consequence, the global attractor $\ms{A}$ in $H_+$ is itself an $(H_+, E_+)$ global attractor for this Oregonator semiflow $\csg$.
\end{proof}

\begin{theorem} \label{T:Infty}
	 The global attractor $\ms{A}$ of the Oregonator semiflow is a bounded subset in $[L^\infty (\gw)]^3$.
\end{theorem}
\begin{proof}
By the $(L^p, L^\infty)$ regularity of the analytic $C_0$-semigroup $\{e^{At}\}_{t\geq 0}$, cf. \cite[Theorem 38.10]{SY02}, one has $e^{At}: [L^p (\gw)]^3 \to [L^\infty (\gw)]^3$ for $t > 0$, and there is a constant $C(p) > 0$ such that
\begin{equation} \label{cp}
	\| e^{At} \|_{\mathcal{L} (L^p, L^\infty)} \leq C(p) \, t^{- \frac{n}{2p}}, \;\; t > 0, \; \; \textup{where} \; n = \textup{dim} \, \gw.
\end{equation}
By the variation-of-constant formula satisfied by the mild solutions, certainly valid for the strong solutions associated with any $g \in \ms{A} \, (\subset E_+)$, we have, for $n \leq 3$,
\begin{equation} \label{mld}
	\begin{split}
	&\|S(t) g\|_{L^\infty} \leq \|e^{At} \|_{\mathcal{L} (L^2,L^\infty)} \|g\| + \int_0^t \|e^{A(t- \sigma)}\|_{\mathcal{L} (L^2, L^\infty)} \| f(S(\sigma)g) \| \, d\sigma \\
	&\leq  C(2) t^{- \frac{3}{4}} \|g\| + \int_0^t C(2) (t- \sigma)^{-\frac{3}{4}} L(\sqrt{K_E}) \|S(\sigma) g\|_E \, d\sigma, \quad t \geq 0,
	\end{split}
\end{equation}
where $C(2)$ is specified in \eqref{cp} with $p = 2$, and $L(\sqrt{K_E})$ is the Lipschitz constant of the nonlinear map $f$ on the closed bounded ball in $E$ centered at the origin and with radius $\sqrt{K_E}$ shown in \eqref{ke}. Also note that $f (0) = 0$. By the invariance of the global attractor $\ms{A}$, surely we have
$$
	\{S(t) \ms{A}: t \geq 0\} = \ms{A} \subset B_0 \, \subset H_+  \quad \textup{and} \quad \{S(t) \ms{A}: t \geq 0\} = \ms{A} \subset B_1 \, \subset E_+.
$$
Then from \eqref{mld} we get
\begin{equation} \label{bdift}
	\begin{split}
	\|S(t) g\|_{L^\infty} &\leq C(2) \sqrt{K_1} t^{- \frac{3}{4}} + \int_0^t C(2) L(\sqrt{K_E}) \sqrt{K_E} (t- \sigma)^{-\frac{3}{4}} \, d\sigma \\
	&= C(2) [\sqrt{K_1}\, t^{- \frac{3}{4}} +  4 L(\sqrt{K_E}) \sqrt{K_E}\, t^{\frac{1}{4}}], \quad \textup{for} \; t > 0.
	\end{split}
\end{equation}
Specifically one can take $t = 1$ in \eqref{bdift} and use the invariance $S(t) \ms{A} = \ms{A}$ to obtain
$$
	\| g \|_{L^\infty} \leq C(2)(\sqrt{K_1} +  4 \sqrt{K_E} L(\sqrt{K_E})), \quad \textup{for any} \; g \in \ms{A}.
$$
Thus the global attractor $\ms{A}$ is a bounded subset in $[L^\infty (\gw)]^3$.
\end{proof}

\section{\textbf{The Existence of Exponential Attractor}}

In this section, we prove the existence of an exponential attractor for the Oregonator semiflow $\csg$ in the invariant cone $H_+$.

\begin{definition}\label{D:exatr}
Let $\mX$ be al Banach space or a closed invariant cone in it and let $\{\Sigma (t)\}_{t\ge 0}$ be a semiflow on $\mX$. A set $\mathscr{E}\subset \mX$ is an exponential attractor for the semiflow $\{\Sigma (t)\}_{t\ge 0}$ in $\mX$, if the following conditions are satisfied:
\begin{enumerate}
\item[(i)]  $\mathscr{E}$ is a nonempty, compact, positively invariant set in $\mX$,
	
\item[(ii)]  $\mathscr{E}$ has a finite fractal dimension, and
	
\item[(iii)]  $\mathscr{E}$ attracts every bounded set $B \subset \mX$ exponentially: there exist positive constants $\mu$ and $C(B)$ which depends on $B$, such that
$$
	\dist_\mX(\Sigma (t)B,\mathscr{E})\leq C(B)e^{-\mu t},\quad\text{for}\,\; t\ge 0.
$$
\end{enumerate}
\end{definition}

The basic theory and construction of exponential attractors were established in \cite{EFNT94} for discrete and continuous semiflows on Hilbert spaces. The existence theory has been generalized to semiflows on Banach spaces in \cite{DN01} and extended to some nonlinear reaction-diffusion equations on unbounded domains. 

%Global attractor, exponential attractor, and inertial manifold are the three major research topics in the theory of infinite dimensional dynamical systems. For a continuous semiflow on a Hilbert space, if all the three objects (a global attractor $\mathscr{A}$, an exponential attractor $\mathscr{E}$, and an inertial manifold $\mathscr{M}$ of the same exponential attraction rate) exist, then the following inclusion relationship holds,
%$$
%	\mathscr{A}\subset\mathscr{E}\subset\mathscr{M}.
%$$

Here we prove the existence of exponential attractor for the Oregonator semiflow by using the following lemma, Lemma \ref{L:EXatr}, which is a modified version of the result shown in \cite[Lemma 6.3]{yY10a}, whose proof was based on the squeezing property \cite{EFNT94, MK05} and the constructive argument in \cite[Theorem 4.5]{MK05}. This lemma provides a way to directly check the sufficient conditions for the existence of an exponential attractor of a semiflow on a positively invariant cone in a Hilbert space. 

\begin{definition}\label{D:sqpy}
For a spectral (orthogonal) projection $P_N$ relative to a nonnegative, self-adjoint, linear operator $\Lambda:D(\Lambda)\to\mathcal{H}$ with a  compact resolvent, which maps the Hilbert space $\mathcal{H}$ onto the $N$-dimensional subspace $\mathcal{H}_N$ spanned by a set of the first $N$ eigenvectors of the operator $\Lambda$, we defined a cone
$$
	\mathscr{C}_{P_N}=\left\{y\in \mathcal{H}:\left\|\left(I-P_N\right)(y)\right\|_{\mathcal{H}}\leq\left\|P_N(y)\right\|_{\mathcal{H}}\right\}.
$$
A continuous mapping $S_*$ satisfies the \emph{discrete squeezing property} relative to a set $B\subset\mathcal{H}$ if there exist a constant $\gk \in(0,1/2)$ and a spectral projection $P_{N}$ on $\mathcal{H}$ such that for any pair of points $y_0,z_0\in B$, if
$$
	S_*\left(y_0\right)-S_*\left(z_0\right)\notin\mathscr{C}_{P_N},
$$
then
$$
	\left\|S_*\left(y_0\right)-S_*\left(z_0\right)\right\|_{\mathcal{H}}\leq\gk \left\|y_0-z_0\right\|_{\mathcal{H}}.
$$
\end{definition}

\begin{lemma}\label{L:EXatr}
Let $\mX$ be a Hilbert space and $X \subset \mX$ be an open cone with the vertex at the origin and $X_c$ be the closure of $X$ in $\mX$. Consider an evolutionary equation
\begin{equation}
\label{eveqig}
	\frac{d\varphi}{dt}+\Lambda \varphi = \Phi (\varphi),\quad t>0,
\end{equation}
where $\Lambda:D(\Lambda)\to \mX$ is a nonnegative, self-adjoint, linear operator with compact resolvent, and $\Phi: \mY=D(\Lambda^{1/2})\to \mX$ is a locally Lipschitz continuous mapping, where $\mY$ is a compactly imbedded subspace of $\mX$. Suppose that the weak solution of \eqref{eveqig} for each initial point $\varphi (0)= \varphi_0\in X_c$ uniquely exists and is confined in $X_c$ for all $t\geq 0$, which turn out to be a strong solution for $t > 0$. All these solutions in $X_c$ form a semiflow denoted by $\{\Sigma (t)\}_{t\ge 0}$. Assume that the following conditions are satisfied:
\begin{enumerate}
\item[\tup{(i)}]  There exist a compact, positively invariant, absorbing set $\mathcal{B}_c$ in $X_c$ with respect to the topology of $\mX$.

\item[\tup{(ii)}]  There is a positive integer $N$ such that the norm quotient $\Gamma (t)$ defined by

\begin{equation}
\label{qn}
	\Gamma (t)=\frac{\left\|\Lambda^{1/2}\left(\varphi_1(t)-\varphi_2(t)\right)\right\|_\mX^2}{\left\|\varphi_1(t)-\varphi_2(t)\right\|_\mX^2}
\end{equation}
for any distinct trajectories $\varphi_1(\cdot)$ and $\varphi_2(\cdot)$ starting from the set $\mathcal{B}_c\backslash\mathscr{C}_{P_N}$ satisfies 
$$
	\frac{d\Gamma}{dt}\le\rho\left(\mathcal{B}_c\right)\Gamma(t),\quad t>0,
$$
where $\rho\left(\mathcal{B}_c\right)$ is a positive constant only depending on $\mathcal{B}_c$.

\item[\tup{(iii)}]  For any given finite $T>0$ and any given $\varphi \in\mathcal{B}_c$, $\Sigma (\cdot)\varphi:[0,T]\to\mathcal{B}_c$ is H\"{o}lder continuous with exponent $\theta=1/2$ and the coefficient of H\"{o}lder continuity, $K_\theta (\varphi):\mathcal{B}_c\to(0,\infty)$, is a bounded function.

\item[\tup{(iv)}]  For any $t\in[0,T]$ where $T>0$ is arbitrarily given, $\Sigma (t)(\cdot):\mathcal{B}_c\to\mathcal{B}_c$ is Lipschitz continuous and the Lipschitz constant $L(t):[0,T]\to(0,\infty)$ is a bounded function.
\end{enumerate}
Then there exists an exponential attractor $\mathscr{E}$ in $X_c$ for this semiflow $\{\Sigma (t)\}_{t\ge 0}$.
\end{lemma}

\begin{theorem}\label{T:exstexp}
Given any positive parameters in the Oregonator system \eqref{eu}--\eqref{ew}, there exists an exponential attractor $\mathscr{E}$ in $H_+$ for the Oregonator semiflow $\csg$.
\end{theorem}
\begin{proof}
By Theorem \ref{T:HEatr}, there exists an $(H_+,E_+)$ global attractor $\mathscr{A}$ for the Oregonator semiflow $\csg$. Consequently, by Corollary 5.7 of \cite{yY10a}, there exists a compact, positively invariant, absorbing set $\mathcal{B}_E$ in $H_+$, which is a bounded set in $E_+$, for this semiflow.

Next we prove that the second condition in Lemma \ref{L:EXatr} is satisfied by this Oregonator semiflow. Consider any two points $g_1(0), g_2(0)\in\mathcal{B}_E$ and let $g_i(t)=\left(u_i(t),v_i(t),w_i (t)\right)$, $i=1,2$, be the corresponding solutions of \eqref{eveq} , respectively. Let $y(t)=g_1(t)-g_2(t)$, $t\geq 0$, where $g_1(0)\ne g_2(0)$. The associated norm quotient of $\left(g_1 (t),g_2 (t)\right)$ is given by
$$
	\Gamma (t)=\frac{\left\|(-A)^{1/2}y(t)\right\|^2}{\|y(t)\|^2},\quad t\geq 0.
$$
For $t > 0$, we can calculate
\begin{equation}
\label{qnest}
\begin{split}
	\frac12 \frac{d}{dt}\Gamma (t) &=\frac1{\|y(t)\|^4}\left[\inpt{(-A)^{1/2}y(t),(-A)^{1/2}y_t}\|y(t)\|^2
    -\|(-A)^{1/2}y(t)\|^2 \inpt{y(t),y_t}\right] \\
         &= \frac1{\|y(t)\|^2} \left[\inpt{(-A)y(t), y_t} - \Gamma (t) \inpt{y(t), y_t}\right] \\
	&=\frac1{\|y(t)\|^2}\inpt{(-A)y(t)-\Gamma (t)y(t),Ay(t)+f\left(g_1(t)\right)-f\left(g_2(t)\right)} \\
	&=\frac1{\|y(t)\|^2}\inpt{(-A)y(t)-\Gamma (t)y(t),Ay(t)+\Gamma (t)y(t)+f\left(g_1(t)\right)-f \left(g_2(t)\right)} \\
	&=\frac1{\|y(t)\|^2}\left[-\|Ay(t)+\Gamma (t)y(t)\|^2-\inpt{Ay(t)+\Gamma (t)y(t),f \left(g_1(t)\right)-f \left(g_2(t)\right)}\right] \\
	&\leq\frac1{\|y(t)\|^2}\left[-\frac12\|Ay(t)+\Gamma (t)y(t)\|^2+\frac12\left\|f \left(g_1(t)\right)-f \left(g_2(t)\right)\right\|^2\right]
\end{split}
\end{equation}
where we used the identity $- \inpt{Ay(t)+\Gamma (t)y(t),\Gamma (t)y(t)}=0$. Note that the compact, positively invariant, $H$-absorbing set $\mathcal{B}_E$ described earlier in this proof is a bounded set in $E_+$ and that $E_+ \hookrightarrow [L^4(\Omega)]^3$ is a continuous imbedding. Hence there is a constant $R>0$ only depending on $\mathcal{B}_E$ such that
\begin{equation}
\label{bR6}
	\|(u,v, w)\|_{L^4(\Omega)}^2\leq R,\quad\text{for any}\;(u,v, w)\in\mathcal{B}_E.
\end{equation}
It is seen that
\begin{equation} \label{fw}
\begin{split}
	\|f (g_1(t))- f(g_2(t))\|^2 = &\, \left\| a_1 (u_1 - u_2) + b_1 (v_1 - v_2) - F(u_1^2 - u_2^2) - G_1 (u_1 v_1 - u_2 v_2)\right\|^2 \\
	&\, +\left\| - b_2 (v_1 - v_2) + c_2 (w_1 - w_2) - G_2 (u_1 v_1 - u_2 v_2)\right\|^2 \\
	&\, + \left\| a_3 (u_1 - u_2) - c_3 (w_1 - w_2) \right\|^2.
\end{split}
\end{equation}
Using the Cauchy-Schwarz inequality, \eqref{pcr}, \eqref{4imb} and \eqref{bR6}, we have
\begin{equation} \bl{fwbd}
	\begin{split}
	&\|f\left(g_1(t)\right)-f\left(g_2(t)\right)\|^2 \\[2pt]
	\leq &\, (4a_1^2 + 4b_1^2 + 3b_2^2 + 3c_2^2 + 2a_3^2 + 2c_3^2) \|y(t)\|^2 + 4F^2 \|u_1 - u_2\|_{L^4}^2 \|u_1 + u_2\|_{L^4}^2  \\[2pt]
	&{} + 8(G_1^2 + G_2^2)\left( \|u_1\|_{L^4}^2 \|v_1 - v_2\|_{L^4}^2 + \|u_1 - u_2\|_{L^4}^2 \|v_2\|_{L^4}^2\right) \\[2pt]
	\leq &\, \ga (4a_1^2 + 4b_1^2 + 3b_2^2 + 3c_2^2 + 2a_3^2 + 2c_3^2) \|\nb y(t)\|^2 + 16 R F^2 \eta^2 \|\nb (u_1 - u_2)\|^2  \\[2pt]
	&{} + 8R (G_1^2 + G_2^2)\eta^2 \left(\|\nb (v_1 - v_2)\|^2 + \|\nb (u_1 - u_2)\|^2 \right) \\[2pt]
	\leq &\,  \left(\ga (4a_1^2 + 4b_1^2 + 3b_2^2 + 3c_2^2 + 2a_3^2 + 2c_3^2) + 16 R F^2 \eta^2 + 8R (G_1^2 + G_2^2)\eta^2\right) \|\nb y(t)\|^2,
	\end{split}
\end{equation}
for $t > 0$. Let
\begin{equation} \label{nr}
	N(R) = 4\ga (a_1^2 + b_1^2 + b_2^2 + c_2^2 + a_3^2 + c_3^2) + 16 R F^2 \eta^2 + 8 R (G_1^2 + G_2^2)\eta^2.
\end{equation}
In view of \eqref{opA}, from \eqref{qnest} and \eqref{fwbd} it follows that
\begin{equation} \label{QN}
	\begin{split}
	\frac{d}{dt}\,\Gamma (t) &\leq \frac{1}{\|y(t)\|^2}\left\|f (g_1(t)) - f (g_2(t))\right\| \leq N(R) \frac{\|\nb y (t)\|^2}{\|y(t)\|^2} \\
	& \leq \rho (\mathcal{B}_E) \frac{\|(-A)^{1/2} y(t)\|^2}{\| y(t) \|^2} = \rho(\mathcal{B}_E) \Gamma (t),\quad t > 0,
	\end{split}
\end{equation}
where $N(R)$ shown in \eqref{nr} is a constant only depending on $R$ which in turn depends on $\mathcal{B}_E$, and
$$
	\rho(\mathcal{B}_E) = \frac{N(R)}{d_0}.
$$ 
where $d_0$ is given in \eqref{dm}. Thus the second condition in Lemma \ref{L:EXatr} is satisfied.

Now check the H\"{o}lder continuity of $S(\cdot)g:[0,T]\to\mathcal{B}_E$ for any given $g\in\mathcal{B}_E$ and any given compact interval $[0,T]$. For any $0\le t_1<t_2\le T$, we get
\begin{equation}
\label{vocs}
\begin{split}
		\left\|S\left(t_2\right)g-S\left(t_1\right)g\right\| &\leq\left\|\left(e^{A\left(t_2-t_1\right)}-I\right)e^{At_1}g\right\|
        +\int_{t_1}^{t_2}\left\|e^{A\left(t_2-\sigma\right)}f(S(\sigma)g)\right\| d\sigma \\
		&+\int_0^{t_1}\left\|\left(e^{A\left(t_2-t_1\right)}-I\right)e^{A\left(t_1-\sigma\right)}f(S(\sigma)g)\right\| d\sigma.
\end{split}
\end{equation}
Since $\mathcal{B}_E$ is positively invariant with respect to the Oregonator semiflow $\csg$ and $\mathcal{B}_E$ is bounded in $E_+$, there exists a constant $K_{\mathcal{B}_E}>0$ such that for any $g\in\mathcal{B}_E$, we have
$$
	\left\| S(t)g \right\|_E\leq K_{\mathcal{B}_E},\quad t\geq 0.
$$
Since $f: E\to H$ is locally Lipschitz continuous, there is a Lipschitz constant $L_{\mathcal{B}_E}>0$ of $f$ relative to this positively invariant set $\mathcal{B}_E$. Moreover, by \cite[Theorem 37.5]{SY02}, for the analytic, contracting, linear semigroup $\{e^{At}\}_{t\ge 0}$, there exist positive constants $N_0$ and $N_1$ such that
$$
	\left\|e^{At}g-g\right\|_H\leq N_0 \, t^{1/2}\|g\|_E,\quad\text{for}\;t\geq 0,\; g \in E,
$$
and 
$$
	\left\|e^{At}\right\|_{\mathcal{L}(H,E)}\leq N_1\,  t^{-1/2},\quad\text{for}\; t>0.
$$
It follows that
$$
	\left\|\left(e^{A\left(t_2-t_1\right)}-I\right)e^{At_1}g\right\|\le N_0 \left(t_2-t_1\right)^{1/2}K_{\mathcal{B}_E}
$$
and
$$
	\int_{t_1}^{t_2}\left\|e^{A\left(t_2-\sigma\right)}f(S(\sigma)g)\right\|d\sigma
    \le\int_{t_1}^{t_2}\frac{N_1 L_{\mathcal{B}_E}K_{\mathcal{B}_E}}{\sqrt{t_2-\sigma}}\, d\sigma
    =2K_{\mathcal{B}_E}L_{\mathcal{B}_E}N_1 \left(t_2-t_1\right)^{1/2}.
$$
Moreover,
\begin{align*}
	\int_0^{t_1}\left\|\left(e^{A\left(t_2-t_1\right)}-I\right)e^{A\left(t_1-\sigma\right)} f(S(\sigma)g)\right\| d\sigma
    &\leq N_0 \left(t_2-t_1\right)^{1/2}\int_0^{t_1} \frac{N_1 L_{\mathcal{B}_E}K_{\mathcal{B}_E}}{\sqrt{t_1-\sigma}}\, d\sigma \\[3pt]
	&=2K_{\mathcal{B}_E}L_{\mathcal{B}_E}N_0 N_1 \sqrt{T}\left(t_2-t_1\right)^{1/2}.
\end{align*}
Substituting the above three inequalities into \eqref{vocs}, we obtain
\begin{equation}
\label{holc}
	\left\|S\left(t_2\right)g-S\left(t_1\right)g\right\|
    \le K_{\mathcal{B}_E}\left(N_0 + 2L_{\mathcal{B}_E}N_1 (1+N_0 \sqrt{T})\right)
    \left(t_2-t_1\right)^{1/2},
\end{equation}
for $0\le t_1<t_2\le T$. Thus the third condition in Lemma \ref{L:EXatr} is satisfied. Namely, for any given $T>0$, the mapping $S(\cdot)g:[0,T]\to\mathcal{B}_E$ is H\"{o}lder continuous with the exponent $\theta = 1/2$ and with a uniformly bounded coefficient independent of $g\in\mathcal{B}_E$.

We can use Theorem 47.8 (specifically (47.20) therein) in \cite{SY02} to confirm the Lipschitz continuity of the mapping $S(t)(\cdot):\mathcal{B}_E\to\mathcal{B}_E$ for any $t\in[0,T]$ where $T>0$ is arbitrarily given. Thus the fourth condition in Lemma \ref{L:EXatr} is also satisfied. Finally, we apply Lemma \ref{L:EXatr} to reach the conclusion that there exists an exponential attractor $\mathscr{E}$ in $H_+$ for the Oregonator semiflow $\csg$.
\end{proof}

\end{document}